\documentclass[lettersize, 11pt]{article}
\pdfoutput=1
\usepackage{etex}
\usepackage{graphicx, bbm, latexsym, amsfonts, epsfig, verbatim, amsmath, algorithmic, color, epsfig, verbatim, enumerate, multirow, caption, pgfkeys}
\usepackage{epstopdf}
\usepackage{url, multicol, color, latexsym, amsfonts, epsfig, verbatim, tikz, array, subeqnarray, tabularx}
\usepgflibrary{arrows}
\usepgflibrary{decorations.pathmorphing}

\usetikzlibrary{arrows,positioning,shapes.geometric}
\usetikzlibrary{decorations.pathmorphing}
\usepackage{booktabs}
\usepackage[title]{appendix}

\usepackage{hyperref}

\usepackage{enumitem}

\usepackage{alphalph}
\usepackage{etoolbox}

\usepackage{subcaption}

\usepackage{algorithm}

\makeatletter
\newalphalph{\aalphalph}[mult]{\alphalph@alph}{26}
\newcommand{\alphalphval}[1]{\@ifundefined{c@#1}{
		\aalphalph{#1}
	}{\aalphalph{\value{#1}}}}
\makeatother

\AtBeginEnvironment{subeqnarray}{%
}

\newtheorem{thm}{Theorem}

\newtheorem{assump}{Assumption}

\newtheorem{proposition}{Proposition}

\newcommand{\ie} {{\em i.e.\/}, }
\newcommand{\eg} {{\em e.g.\/}, }

\newcommand{\jh}[1]{\textcolor{black}{#1}}
\newcommand{\jho}[1]{\textcolor{black}{#1}}
\newcommand{\jhw}[1]{\textcolor{black}{#1}}
\newcommand{\kz}[1]{\textcolor{black}{#1}}

\newcommand{\0}{\kern-1.2pt\vec{\kern1.2pt 0}}

\definecolor{darkblue}{RGB}{65,105,225}
\hypersetup{colorlinks,breaklinks,
            linkcolor=darkblue,urlcolor=darkblue,
            anchorcolor=darkblue,citecolor=darkblue}

\def\Box {\vrule height5pt width5pt depth0pt}
\def\beginproof{\par\noindent {\bf \em Proof:}\ \ }
\def\endproof{\hskip .5cm \Box \vskip .5cm}

\def\R{{\mathbb{R}}}

\def\E{{\mathbb{E}}}

\def\P{{\mathbb{P}}}

\def\D{{\mathcal D}}

\def\Qe{{\mathcal Q}}

\def\ld{\lambda}

\allowdisplaybreaks

\begin{document}
	\title{ {\bf A Study of Distributionally Robust Multistage Stochastic Optimization}\\[3mm] }	
	\author{\normalsize \bf Jianqiu Huang, Kezhuo Zhou, and Yongpei Guan
	\\ \\
	{\small Department of Industrial and Systems Engineering}\\
	{\small University of Florida, Gainesville, FL 32611, USA}\\
	{\small Emails: jianqiuhuang@ufl.edu, zhoukezhuo@ufl.edu, and guan@ise.ufl.edu}\\
	}
	\date{}
	
	\maketitle
	
	\vspace{-0.3in}

\begin{abstract}
	\setlength{\baselineskip}{18pt}
In this paper, we focus on a \kz{data-driven} risk-averse multistage stochastic programming (RMSP) model \kz{considering} distributional \kz{robustness}. We optimize the RMSP over the worst-case distribution within an ambiguity set \jho{of}  probability distributions \jho{constructed directly from} historical data samples. The proposed RMSP is intractable \kz{due to the multistage nested minimax structure \jhw{in} its objective function}, so we \kz{reformulate it into \jho{a deterministic} equivalent that contains} a \kz{series of} convex combination of expectation and \jho{conditional value at risk} (CVaR), \kz{which can be solved by a} \jho{customized} stochastic dual dynamic programming (SDDP) algorithm \kz{in this paper}. As the size of collected data samples increases to infinity, \kz{we show the consistency of the} RMSP with distributional \kz{robustness} to the traditional multistage stochastic programming. In addition, 
to test the computational performance \kz{of our proposed model and algorithm}, we \kz{conduct} numerical experiments for a risk-averse hydrothermal scheduling problem, the results of which demonstrate the \kz{effectiveness} of our RMSP \kz{framework}.
	
	\vspace{0.10in}
	
\end{abstract}

\setlength{\baselineskip}{22pt}
\section{Introduction}\label{sec:introdutcion}
The multistage stochastic program (MSP) has been widely studied in literature for providing multi-period optimal decisions  under uncertainty,  \jho{since MSP} can be naturally adopted to model various real-life applications with periodical decisions, \eg hydrothermal scheduling \cite{pereira1991multi}, power system operations \cite{takriti1996stochastic,sen2006stochastic,pan2016strong}, transportation \cite{powell1995stochastic}, and supply chain planning \cite{santoso2005stochastic}.
For MSP models, the sequential decisions are made \kz{depending on the realization of} stochastic parameters, which are assumed to follow some known probability distributions. 
Once a decision has been made for the previous period, an observation of the stochastic parameters for the current period  becomes available and then the corresponding decision   will be made considering  future uncertainty. The objective of MSP is to minimize the total expected costs incurred by the decisions over the planning horizon subject to a series of constraints, \eg nonanticipativity constraints and other model\kz{ing} constraints. Readers are referred to \cite{kall1994stochastic}, \cite{shapiro2009lectures}, \cite{birge2011introduction}, and \cite{pflug2014multistage} for more detailed  MSP structures and properties.

\jh{MSP models are generally computationally intractable, since each decision making depends on not only the past parameter realization, but also decisions made in each previous stage}.
To simplify the model, the stochastic parameters are assumed to be discrete random variables and modeled via scenario trees, \jho{leading to a deterministic equivalent model of MSP, where the corresponding expectations in the objective function convert to finite sums}. 
Two types of decomposition algorithms are proposed to solve MSP, \ie scenario-based methods where the sample of realizations is fixed, and sampling-based methods where the sample of realizations is obtained iteratively.  Scenario-based methods use \jho{a} small set of realizations \jho{from the complete sample space} to approximately solve the original program, like diagonal quadratic approximation \cite{mulvey1992diagonal}, Lagrangian decomposition \cite{rosa1996augmented}, L-shaped methods \cite{van1969shaped,birge1985decomposition}, and scenario aggregation methods \cite{rockafellar1991scenarios}. On the other hand, sampling-based methods iteratively draw a subset of realizations from the complete sample space, where statistical bounds are utilized to create convergence criteria, \eg stochastic dual dynamic programming (SDDP) -- a Monte Carlo sampling-based method  \cite{pereira1991multi,shapiro2011analysis,chen1999convergent,donohue2006abridged,linowsky2005convergence,philpott2008convergence,kozmik2015evaluating}, 
the stochastic decomposition with extension to the multistage case \cite{higle2009stochastic,sen2014multistage},  and progressive hedging \cite{rockafellar1991scenarios,takriti1996stochastic}. \jho{These two types of methods are integrated together in \cite{rebennack2016combining}.}
\jhw{Moreover, multistage distances are introduced and utilized to solve MSP models approximately \cite{pflug2012distance,pflug2014multistage}.}

\jho{Over the last few decades}, risk-averse multistage stochastic programming (RMSP) has been attracting significant attentions, \jho{due to its advantages over MSP on modeling certain applications where \jho{low} probability events have high impact}.
The classical MSP focuses on the average behavior, however, there are many real-life applications where \jho{low} probability events have high impact. 
To handle this issue, various risk measures are designed to remedy \jho{the limitation} of the traditional expectation operator in the \jho{MSP objective function}, \ie RMSP uses a convex combination of a expectation and a risk measure \cite{rockafellar2000optimization,krokhmal2011modeling}. Seminal novel works  exploring the risk measure properties \jho{for RMSP include} coherent risk measures \cite{artzner1999coherent}, time-consistent risk measures \cite{ruszczynski2006conditional,ruszczynski2010risk}, and regularity of risk measures \cite{rockafellar2013fundamental}. Lagrangian relaxation is first proposed to solve the RMSP models \cite{eichhorn2005polyhedral,collado2012scenario}, and later \jh{come} the advanced nested L-shaped decomposition algorithm\kz{s} \cite{ahmed2006convexity,miller2011risk,philpott2013solving,guigues2014sddp,kozmik2015evaluating}.

Both MSP and RMSP  rely on  known probability distributions, which is generally not practical due to the ambiguity of the probability distributions in real-life applications. Hence, ambiguity sets of probability distributions are proposed to model all possible probability distributions \kz{within certain range} based on a series of historical data samples. 
Several types of ambiguity sets have been proposed, like moment-based  ambiguity sets \cite{shapiro2004class, delage2010distributionally,scarf1958min,vzavckova1966minimax,zymler2013distributionally, zymler2013worst, mehrotra2014models}, divergence-measure-based ambiguity sets \cite{klabjan2013robust,  calafiore2007ambiguous,ben2013robust,love2015phi}, and other metric-based ambiguity sets \cite{erdougan2006ambiguous,pflug2007ambiguity,mehrotra2014models,jiang2015risk}. 
Seminal works also construct ambiguity sets based on  the relation between risk measures and robust optimization models \cite{bertsimas2009constructing,natarajan2009constructing,ben2010soft}.

Our contribution in this paper is that  we first propose an RMSP formulation by integrating the general MSP model and the distributional ambiguity sets \kz{with  $ L_\infty $-norm}, which optimizes the total expected costs over the worst-case distribution within the ambiguity set. The ambiguity sets are constructed directly from the historical data sample\jhw{s}. We then derive \kz{an} equivalent reformulation of RMSP,  where the objective function is replaced by a convex combination of an expected cost and a \jho{conditional value at risk} (CVaR). A significant advantage of our reformulation is that it gets rid of the nested multistage minimax \kz{structure from} the original objective function, leading to a tractable MSP based on \kz{the} reference distributions \jho{constructed} \kz{from historical data}. We prove that   optimal solutions and \kz{objective} values of our RMSP \jho{with distributional robustness} converge to those of risk-netural MSP as the size of data samples increases to infinity. In addition, we customize and analyze the SDDP algorithm to solve \kz{our proposed} RMSP.  Furthermore, we apply RMSP with distributional \kz{robustness} to the hydrothermal scheduling problem, and implement corresponding computational experiments whose results verify the convergence of our method.

The following sections of this paper are organized as follows. In Section~\ref{sec:reformulation}, we present our RMSP model with distributional \kz{robustness} and reformulate it \jh{into} a tractable \kz{equivalent}, \kz{and show the consistency of} our RMSP \kz{model}. In Section~\ref{sec:sddp}, we customize the SDDP algorithm to solve our RMSP, followed by a convergence analysis of our SDDP algorithm. 
Next, we apply our RMSP to a risk-averse hydrothermal scheduling problem in Section~\ref{sec:compute}, and provide the corresponding computational experiment results. Finally we conclude our research in Section~\ref{sec:conclusion}.

\section{Risk-Averse Multistage Stochastic Program}\label{sec:reformulation}
In this section, we present the reformulation \kz{of} RMSP based on ambiguity sets \kz{with $ L_{\infty} $-norm} $ \D^t,  \forall\ t = 2,\dots,T $,  and then provide the \jho{corresponding} convergence analysis. 

\kz{At each stage $ t = 2, 3, \cdots, T $, we denote the sample space of the stochastic parameters $ \xi_t $ as $ \Omega_t = \{\xi_t^1, \xi_t^2, \cdots, \xi_t^R\}$}. Similarly to that of \cite{jiang2015risk}, the ambiguity set $ \D^t $ \kz{for possible probability distributions $ f = (f^1, f^2, \cdots, f^R) $} can be constructed in a data-driven way with \kz{two kinds of} representations as follows:
\begin{equation}
\mathcal{D}^t = \Bigg\{ f\geq 0: \kz{||f - f_0 ||_\infty \leq d_t},   \ \sum_{r=1}^{R} f^r=1	\Bigg\}, \label{formula1}
\end{equation}
or
\begin{equation}
\mathcal{D}^t = \Bigg\{ f\geq 0: \jho{d_1^t \leq ||f/f_0 ||_\infty \leq d_2^t ,}  \ \sum_{r=1}^{R} f^r=1	\Bigg\}, \label{formula2}
\end{equation}
\kz{where $ f_0 $ represents the reference distribution that can be established from the empirical distribution with historical data, \jh{$ d_t $ is a tolerance that decreases as the data size $ N_t $ increases}, \jho{and $ d_1^t/ d_2^t $ are tolerance parameters that increase/decrease to $ 1 $ as the data size $ N_t $ grows to infinity}. 
The decision rules of \jho{$ d_t$, $ d_1^t $, and $ d_2^t $} are beyond the scope of this paper and readers are referred to \cite{jiang2015risk} for more details in this part. 
\jho{It is obvious that representations \eqref{formula1} and \eqref{formula2} can be transformed to each other by selecting proper parameters $ d_t $, $ d_1^t $ and $ d_2^t $.}
\jho{Therefore}, we focus on the first approach to construct the distributional ambiguity set $ \mathcal{D}^t $ in this paper.}
\jh{We denote the lower and upper bound of the ambiguity set as $ f_{\ell} = f_0 -d_t $ and $f_{u} = f_0 +d_t $, respectively}. 
At each stage $ t =2,\dots,T $, we let $ \P_t^\ell$ and $ \P_t^{u-\ell} $ represent the probability measures induced by $ f_{\ell}$ and  $ f_u -f_{\ell} $, and we assume $ P_t^{\ell} := \sum_{r=1}^{R}f_{\ell}^r \in (0,1) $ and $ P_t^u := \sum_{r=1}^{R}f_{u}^r >1 $  to avoid trivial cases. 
In the following, we reformulate the worst-case expected cost over $ \D^t $ at each stage $ t $ \jh{into} a convex combination of an expected cost and a CVaR. Finally, we show that both the set of optimal solutions and the objective value of RMSP converge to \jho{those of} risk-neutral MSP as the data sample size grows to infinity, respectively.

\subsection{Equivalent Reformulation}\label{subsec:reformulate}
In this subsection, we develop the equivalent reformulation of the following nested RMSP formulation:  
\begin{equation}\label{formulation:origin}
\min_{A_1x_{1}\geq b_1} c_{1}x_{1} +  \sup_{\P_2 \in \mathcal{D}^2}\E_{\P_2}\Big[ \min_{A_2x_2 \geq b_2 - B_2x_1} c_2x_2 + \cdots + \sup_{\P_T \in \mathcal{D}^T}\E_{\P_T}[ \min_{A_Tx_T \geq b_T -B_Tx_{T-1}  } c_Tx_T  ]  \Big],
\end{equation}
where vectors $ c_t, b_t $, and matrices $ A_t, B_t $ are assumed to be stagewise independent random variables forming the stochastic data process \kz{$ (c_t(\xi_t), b_t(\xi_t), A_t(\xi_t), B_t(\xi_t)) $} for $ t=2,\dots, T $.

\begin{assump} \label{assump_bd}
	\kz{The RMSP has compact feasible set and relatively complete recourse. \jh{In addition, the} recourse function \jh{at each stage} is bounded for each decision $ x $.} 
\end{assump}
Due to the stagewise independence of the data process, formulation~\eqref{formulation:origin} can reformulated as a series of dynamic programming equations.
Starting from the last stage $ T $, we define $ Q_T(x_{T-1}(\xi_{T-1}), \xi_{T}   ) $ as the optimal value of the last stage program as follows:
\begin{align*}
Q_T(x_{T-1}(\xi_{T-1}), \xi_{T}   ) = \hspace{-0.5in} &&\min_{x_{T}} \quad & c_T(\xi_{T})x_T(\xi_{T}) \notag \\
\hspace{-0.5in} &&s.t. \quad & A_T(\xi_{T})x_T(\xi_{T}) \geq b_T(\xi_{T}) - B_T(\xi_{T})x_{T-1}(\xi_{T-1}).
\end{align*}

Backward to stage $ T-1 $, we have that $ Q_{T-1}(x_{T-2}(\xi_{T-2}), \xi_{T-1}   ) $ is equal to the optimal value of the program
\begin{align}\label{formulation:T}
\hspace{-0.16in}Q_{T-1}(x_{T-2}(\xi_{T-2}), \xi_{T-1}) = \hspace{-0.15in} &&\min_{x_{T-1}}\: & c_{T-1}(\xi_{T-1})x_{T-1}(\xi_{T-1}) +  \sup_{\P_T\in \mathcal{D}^T}\E_{\xi_T\sim \P_T}[Q_T( x_{T-1}(\xi_{T-1}), \xi_T )] \notag \\
\hspace{-0.1in}&& s.t.\: & A_{T-1}(\xi_{T-1})x_{T-1}(\xi_{T-1}) \geq b_{T-1}(\xi_{T-1}) - B_{T-1}(\xi_{T-1})x_{T-2}(\xi_{T-2}),
\end{align}
and the cost-to-go function $ \Qe_T( x_{T-1}(\xi_{T-1}) ) $ for the worst-case expectation \jh{$ \sup_{\P_T\in \mathcal{D}^T}\E_{\xi_T\sim \P_T}[ \allowbreak Q_T( x_{T-1}(\xi_{T-1}),\allowbreak \xi_T )] $} can be defined and then reformulated as follows:
\begin{align}\slabel{eqn:Qe_T}
&\Qe_T( x_{T-1}(\xi_{T-1}) )  \equiv  \sup_{\P_T\in \mathcal{D}^T}\E_{\xi_T\sim \P_T}[Q_T( x_{T-1}(\xi_{T-1}), \xi_T )] 	\notag	\\
&=  P_T^\ell\E_{\xi_T\sim \P_T^\ell}[Q_T(x_{T-1}(\xi_{T-1}), \xi_{T}   )] + (1-P_T^\ell)\mbox{CVaR}_{(P_T^u-1)/(P_T^u-P_T^\ell)}^{\xi_T\sim \P_T^{u-\ell}}[Q_T( x_{T-1}(\xi_{T-1}), \xi_T )]			\label{eqn:risk-measure}	\\
& = P_T^\ell\E_{\xi_T\sim \P_T^\ell}[Q_T(x_{T-1}(\xi_{T-1}), \xi_{T}   )] + \inf_{ u_{T-1} \in \R }  \Big\{  (1-P_T^\ell)u_{T-1}(\xi_{T-1}) \notag  \\
&+ (P_T^u-P_T^\ell)\E_{\xi_T\sim \P_T^{u-\ell}}[Q_T(x_{T-1}(\xi_{T-1}), \xi_{T}) - u_{T-1}(\xi_{T-1})]^+    \Big\}, \label{eqn:cvar-def}
\end{align}
where equality~\eqref{eqn:risk-measure} holds due to Theorem 6 in \cite{jiang2015risk} and equality~\eqref{eqn:cvar-def} holds because of the CVaR definition.

Thus, $ Q_{T-1}(x_{T-2}(\xi_{T-2}), \xi_{T-1}   ) $ can be reformulated by substituting \jh{Equation~\eqref{eqn:cvar-def}} into \eqref{formulation:T}.
\begin{align}
\min_{x_{T-1}, u_{T-1}} \quad & c_{T-1}(\xi_{T-1})x_{T-1}(\xi_{T-1}) + (1-P_T^\ell )u_{T-1}(\xi_{T-1}) + P_T^\ell \E_{\xi_T\sim \P_T^\ell }[Q_T(x_{T-1}(\xi_{T-1}), \xi_{T}   )] \notag \\
&+ (P_T^u-P_T^\ell )\E_{\xi_T\sim \P_T^{u-\ell}}[Q_T(x_{T-1}(\xi_{T-1}), \xi_{T}) - u_{T-1}(\xi_{T-1})]^+ \notag \\
s.t. \qquad & A_{T-1}(\xi_{T-1})x_{T-1}(\xi_{T-1}) \geq b_{T-1}(\xi_{T-1}) - B_{T-1}(\xi_{T-1})x_{T-2}(\xi_{T-2}).
\end{align}

By repeating this process backward, we can write dynamic programming equations for each stage $ t=2,\dots,T $ as
\begin{align}\label{formulation:t}
Q_t(x_{t-1}(\xi_{t-1}), \xi_t) =& \inf_{x_{t }, u_{t}} \Big\{  c_{t }(\xi_{t })x_{t }(\xi_{t }) + (1-P_{t+1}^\ell )u_{t}(\xi_t) + \Qe_{t+1}( x_{t }(\xi_{t }), u_{t}(\xi_{t }) ): \notag \\
& A_{t }(\xi_{t })x_{t }(\xi_{t }) \geq b_{t }(\xi_{t }) - B_{t }(\xi_{t })x_{t-1}(\xi_{t-1}) \Big\}, 
\end{align}
where the cost-to-go function $ \Qe_{t+1}( x_{t }(\xi_{t }), u_{t}(\xi_{t }) ) $ is defined as   
\begin{align}\label{eqn:cost-to-go}
\Qe_{t+1}( x_{t }(\xi_{t }), u_{t}(\xi_{t }) ) =& P_{t+1}^\ell \E_{\xi_{t+1}\sim \jh{\P_{t+1}^\ell} }[Q_{t+1}(x_{t }(\xi_{t }), \xi_{t+1}   )] \notag \\
&+ (P_{t+1}^u-P_{t+1}^\ell )\E_{\xi_{t+1}\sim \jh{\P_{t+1}^{u-\ell}}}[Q_{t+1}(x_{t }(\xi_{t }), \xi_{t+1}) - u_{t}(\xi_t) ]^+ ,
\end{align}
with $ \Qe_{T+1}(\cdot) = 0 $ and $ P_{T+1}^u = P_{T+1}^\ell  = 1 $.

Finally, the reformulated program at the first stage is described as follows:
\begin{align}\label{formulation:t=1}
\min_{x_{1}, u_{1}} \qquad & c_{1}x_{1}  +  (1-P_{2}^\ell )u_{1 } + \Qe_{2}( x_{1 }, u_{1} ) \notag \\
s.t. \qquad & A_{1} x_{1}  \geq b_{1}.
\end{align}

\subsection{Convergence Analysis}\label{subsec:converge}
In this section, we \jho{show the consistency of RMSP with distributional robustness by analyzing} the convergence property of \jh{RMSP} as the size of historical data samples increases to infinity. We find that when the size of historical data samples for constructing ambiguity sets at each stage goes to infinity, both the optimal objective value and the set of optimal solutions for \jh{RMSP} converge to the counterparts of MSP under true \kz{while unknown} distribution.
\kz{We define the following \jho{notations} for our proof.} We let $ z(0) $ denote the optimal objective value and $ U(0) $ denote the set of optimal solutions \kz{for the MSP under  \jho{the} true distribution. Similarly, we let $ \hat{z}(0) $ denote the optimal objective value and $ \hat{U}(0) $ denote the set of optimal solutions for the MSP under \jho{the} reference distribution. We extend the notations to RMSP by denoting $ \hat{z}(d_2(N_2),\ldots,d_T(N_T)) $ as the optimal objective value with data size $ N_t $ for each corresponding stage and $ \hat{U}(d_2(N_2),\ldots,d_T(N_T)) $ as the optimal solutions under the same setting.} 

For \jh{MSP}, the recourse functions are overlapped in a nest\kz{ed} structure and only in the last stage do we have a \jho{closed-form} recourse function for each scenario. Moreover, the number of decision variables in stage $ t $ grows in an exponential rate of stage $ t $. Thus, in order to better show our analysis of the convergence property, we denote $ x_t(\xi_t) $ \jh{as} the decision in stage $ t $ for \kz{observation} $ \xi_t $, $ x_t $ \jh{as} the set of decisions in stage $ t $ for all \kz{observations} in stage $ t $, and $ x $ \jh{as} the set of decisions for all stages and scenarios. \jho{Besides,} we let $ N=\min\{N_2,N_3,\ldots,N_T\}$ \jho{and} $ d=\max\{d_2,d_3,\ldots,d_T\} $. \jho{With the above notations,} we provide our conclusion on the convergence analysis for \jh{RMSP}.
\begin{proposition} \label{reflimit}
	As the size of data sample $ N $ goes to $ \infty $, $ \hat{z}(0)\rightarrow z(0) $. Furthermore, $ \hat{U}(0) $ converges to $ U(0) $, i.e., $ \lim_{N \rightarrow \infty}\sup_{x\in \hat{U}(0)}||x-U(0)||=0 $.
\end{proposition}
\begin{proof}
	Along with the proof in this paper, we will use $ N\rightarrow \infty $ and $ d\rightarrow 0 $ interchangeably. Note that there exist underlying reference distributions along all stages for each $ \hat{z}(0) $, so we may use the notation $ \hat{z}_{\P_N}(0) $ when we need to emphasize it. We follow the same \kz{notation} rule for $ z(0) $ and throughout the proof in this paper. We also let $ \hat{h}(x,d_2,\ldots,d_T) $ represent the objective value \kz{for RMSP} \jho{corresponding to} solution $ x $ and tolerance\jho{s} $ d_2,\ldots,d_T $, and similarly $ h(x) $ \kz{for MSP with true distribution} under solution $ x $.
	\begin{equation}
	\jho{\hat{z}(0) = \hat{h}(\hat{x}^*,0) \leq \hat{h}(x^*,0)}, \label{eqn1}
	\end{equation}
	where $ \hat{x}^* $ represents an optimal solution to RMSP under corresponding tolerance, 
	\jho{and} $ x^* $ represents an optimal solution to \kz{MSP} under true distribution. 
	\jho{T}aking upper limit with respect to the size of historical data on both sides \jho{of Inequality~\eqref{eqn1},} we have
	\begin{eqnarray}
		&& \limsup_{N\rightarrow \infty} \hat{z}(0) = \limsup_{N\rightarrow \infty} \hat{h}(\hat{x}^*,0) \leq \limsup_{N\rightarrow \infty} \hat{h}(x^*,0) = \lim_{N\rightarrow \infty} \hat{h}(x^*,0) = h(x^*)=z(0), \slabel{limsup}
	\end{eqnarray}
\jho{where the second equality holds}
	because of the following reasons: \kz{once the decision variables $ x $ are fixed, the objective function is a polynomial of parameters $ f^r_t $; the reference distribution for each stage $ t $ converges weakly to the true distribution, so the limit of the objective value exists and is consistent with that under true distribution.}
	
	\jho{Next, we} show that $ \{\hat{z}_{\P_N}(0)\} $ converges. Under the \kz{Assumption \ref{assump_bd}} it is easy to see that $ \{\hat{z}_{\P_N}(0)\} $ belongs to a bounded set. So if $ \{\hat{z}_{\P_N}(0)\} $ does not converge, we can find two subsequences that converge to different value\kz{s}, say $ z_1 $ and $ z_2 $, i.e., $ \hat{z}_{\P_{N_t}}(0) \rightarrow z_1,\hat{z}_{\P_{N_s}}(0) \rightarrow z_2 $ and $ z_1 \neq z_2 $. Note that the corresponding optimal solutions  $ \{\hat{x}_{N_t}^*\} $ and $ \{\hat{x}_{N_s}^*\} $ are bounded due to the assumption of compact feasible region, there exists subsequences of the two \kz{series of} optimal solutions that converge respectively. For notation brevity\jho{,} we still denote the two subsequences as  $ \{\hat{x}_{N_t}^*\} $ and $ \{\hat{x}_{N_s}^*\} $, and $ \hat{x}_{N_t}^* \rightarrow x_1, \hat{x}_{N_s}^*\rightarrow x_2 $, where $ x_1 $ and $ x_2 $ are two feasible solution. Then we have
	\begin{subeqnarray}
		&& z_1 = \lim_{t\rightarrow \infty} \hat{z}_{\P_{N_t}}(0) = \lim_{t\rightarrow \infty}\hat{h}(\hat{x}_{N_t}^*,0) = \lim_{t\rightarrow \infty}\hat{h}(x_1,0) = h(x_1) \geq z(0)\jh{,} \nonumber\\
		&& z_2 = \lim_{s\rightarrow \infty} \hat{z}_{\P_{N_s}}(0) = \lim_{s\rightarrow \infty}\hat{h}(\hat{x}_{N_s}^*,0) = \lim_{s\rightarrow \infty}\hat{h}(x_2,0) = h(x_2) \geq z(0)\jh{.} \nonumber
	\end{subeqnarray}

	But we have $ z_1 \leq z(0) $ and  $ z_2 \leq z(0) $ due to \eqref{limsup}, so $ z_1 = z(0) =z_2 $, which is a contradiction. Thus, $ \{\hat{z}_{\P_N}(0)\} $ converges and we have
	\begin{eqnarray}
		&& \lim_{N\rightarrow \infty} \hat{z}(0) = \limsup_{N\rightarrow \infty} \hat{z}(0) = z(0). \nonumber
	\end{eqnarray}
	
	\jho{Finally}, we prove the convergence property of $ \hat{U}(0) $ to $ U(0) $ by contradiction. \kz{Supposing} that $ \sup_{x\in \hat{U}(0)}||x-U(0)|| $ does not converge to zero as $ N $ grows to infinity, there exists a positive number $ \epsilon_0 $ and a sequence of optimal solutions $ \{\hat{x}_{N_k}^*\} $ such that $ ||\hat{x}_{N_k}^*-U(0)|| > \epsilon_0 $ for all $ k $. Following the same idea \jho{as} above, we have a subsequence of the optimal solutions $ \{\hat{x}_{N_k}^*\} $ that converges. For notation brevity, we still denote the subsequence as  $ \{\hat{x}_{N_k}^*\} $ and  $ \hat{x}_{N_k}^* \rightarrow \bar{x} $, where $ \bar{x} $ is a feasible solution. 
	Then \jho{we have} $ \lim_{k\rightarrow \infty}\hat{h}(\hat{x}_{N_k}^*,0) = \lim_{k\rightarrow \infty}\hat{h}(\bar{x},0) = h(\bar{x}) $.
	Since $ \{\hat{x}_{N_k}^*\} $ is a sequence of optimal solutions\jho{,} we conclude that $ h(\bar{x}) = z(0) $ and accordingly $ \bar{x} \in U(0)$. However, as $ \hat{x}_{N_k}^* \rightarrow \bar{x} $ and $ ||\hat{x}_{N_k}^*-U(0)|| > \epsilon_0 $, we have $ ||\bar{x}-U(0)|| \geq \epsilon_0 > 0$, which is a contradiction. This completes the proof.
\end{proof}

\begin{thm}
	For all $ t\geq 2 $, as the size of historical data samples $ N_t $ increases to $ \infty $, the distance tolerance $ d_t(N_t) \rightarrow 0$, \kz{$ \max_{\xi_t \in \Omega_t} |f^t(\xi_t)-f_0^t(\xi_t)|\rightarrow 0$}, and $ \hat{z}(d_2(N_2),\ldots,d_T(N_T))\rightarrow z(0) $. Furthermore, $ \hat{U}(d_2(N_2),\ldots,d_T(N_T)) $ converges to $ U(0) $, i.e., $ \lim_{N_2,\ldots,N_T \rightarrow \infty}\sup_{x\in U(d_2(N_2),\ldots,d_T(N_T))}||x-U(0)||=0 $.
\end{thm}
\begin{proof}
	To prove that $ \lim_{N\rightarrow \infty}\hat{z}(d_2(N_2),\ldots,d_T(N_T)) = z(0) $, it is enough to show that
	\begin{eqnarray}
		\limsup_{N\rightarrow \infty}\hat{z}(d_2(N_2),\ldots,d_T(N_T)) \leq z(0), \slabel{uppgoal}
	\end{eqnarray}
	and
	\begin{eqnarray}
		\liminf_{N\rightarrow \infty}\hat{z}(d_2(N_2),\ldots,d_T(N_T)) \geq z(0).  \slabel{lowgoal}
	\end{eqnarray}
	First, since $ \hat{h}(\hat{x}_{N,d}^*,d_2(N_2),\ldots,d_T(N_T)) \geq \hat{h}(\hat{x}_{N,d}^*,0) \geq \hat{h}(\hat{x}_{N,0}^*,0) $, where $ \hat{x}_{N,d}^* $ and $ \hat{x}_{N,0}^* $ represent the optimal solution to corresponding \jh{RMSP} under tolerance $ (d_2(N_2),\ldots,d_T(N_T)) $ and $ 0 $ respectively, by taking lower limit, we have
	\begin{subeqnarray}
		&& \liminf_{N\rightarrow \infty}\hat{z}(d_2(N_2),\ldots,d_T(N_T)) \nonumber \\
		& = &\liminf_{N\rightarrow \infty}\hat{h}(\hat{x}_{N,d}^*,d_2(N_2),\ldots,d_T(N_T)) \slabel{lowlim1}\\
		& \geq & \liminf_{N\rightarrow \infty} \hat{h}(\hat{x}_{N,0}^*,0) \slabel{lowlim2}\\
		& = & \lim_{N\rightarrow \infty} \hat{h}(\hat{x}_{N,0}^*,0) \slabel{lowlim3}\\
		& = & \lim_{N\rightarrow \infty} \hat{z}(0)\slabel{lowlim4} \\
		& = & z(0), \slabel{lowlim5}
	\end{subeqnarray}
	where \eqref{lowlim1} holds by definition, \eqref{lowlim2} holds because of the inequality we just provide, \jho{and} \eqref{lowlim3} to \eqref{lowlim5} hold because of Proposition \ref{reflimit}. Thus, inequality \eqref{lowgoal} is proved.
	
	Second, we have $ \hat{h}(\hat{x}_{N,d}^*,d_2(N_2),\ldots,d_T(N_T)) \leq  \hat{h}(x,d_2(N_2),\ldots,d_T(N_T)) $ for any feasible solution $ x $. More specifically, it can be written briefly as 
	\begin{align} \label{uppineqn}
		  & \min \; c_{1}x_{1} +  \sup_{\P_2 \in \mathcal{D}\kz{^2}}\E_{\P_2}\Big[ \min c_2 x_2 + \cdots + \sup_{\P_T \in \mathcal{D}\kz{^T}}\E_{\P_T}[ \min  c_T x_T ]  \Big] \notag \\
		 \leq \quad & c_{1}x_{1} +  \sup_{\P_2 \in \mathcal{D}\kz{^2}}\E_{\P_2}\Big[c_2 x_2 + \cdots + \sup_{\P_T \in \mathcal{D}\kz{^T}}\E_{\P_T}[  c_T x_T ]  \Big] 
	\end{align}
	for all feasible $ x $. Taking upper limit on both sides of inequality \eqref{uppineqn}, we can obtain that 
	\begin{subeqnarray} \label{uppineqn2}
		&& \limsup_{N\rightarrow \infty}\hat{z}(d_2(N_2),\ldots,d_T(N_T)) \nonumber \\
		& = & \limsup_{N\rightarrow \infty}\hat{h}(\hat{x}_{N,d}^*,d_2(N_2),\ldots,d_T(N_T)) \slabel{upplim1} \\
		& \leq & \limsup_{N\rightarrow \infty} \Big\{ c_{1}x_{1} +  \sup_{\P_2 \in \mathcal{D}\kz{^2}}\E_{\P_2}\Big[c_2 x_2 + \cdots + \sup_{\P_T \in \mathcal{D}\kz{^T}}\E_{\P_T}[  c_T x_T ]  \Big] \Big \} \slabel{upplim2} \\
		& = & \limsup_{d_T\rightarrow 0} \cdots \limsup_{d_2\rightarrow 0}\Big\{ c_{1}x_{1} +  \sup_{\P_2 \in \mathcal{D}\kz{^2}}\E_{\P_2}\Big[c_2 x_2 + \cdots + \sup_{\P_T \in \mathcal{D}\kz{^T}}\E_{\P_T}[  c_T x_T ]  \Big] \Big \}  \slabel{upplim3} \\
		& = & \limsup_{d_T\rightarrow 0} \cdots \limsup_{d_3\rightarrow 0} \Big\{ c_{1}x_{1} + \limsup_{d_2\rightarrow 0} \sup_{\P_2 \in \mathcal{D}\kz{^2}}\E_{\P_2}\Big[c_2 x_2 + \cdots + \sup_{\P_T \in \mathcal{D}\kz{^T}}\E_{\P_T}[  c_T x_T ]  \Big] \Big \} \slabel{upplim4} \\
		& = & \limsup_{d_T\rightarrow 0} \cdots \limsup_{d_3\rightarrow 0} \Big\{ c_{1}x_{1} + \E_{\P_2}\Big[c_2 x_2 + \cdots + \sup_{\P_T \in \mathcal{D}\kz{^T}}\E_{\P_T}[  c_T x_T ] \Big] \Big \} \slabel{upplim5} \\
		& = & \cdots \slabel{upplim6} \\
		& = &  c_{1}x_{1} +  \E_{\P_2}\Big[c_2 x_2 + \cdots + \E_{\P_T}[  c_T x_T ]  \Big], \slabel{upplim7}
	\end{subeqnarray}
	where \eqref{upplim1} holds by definition, \eqref{upplim2} holds because of inequality \eqref{uppineqn}, \eqref{upplim3} holds because $ N\rightarrow \infty $ is equivalent to $ d_2(N_2),\ldots,d_T(N_T) \rightarrow 0  $, \eqref{upplim4} holds because here we only consider the limit when $ d_2 $ goes to zero, \eqref{upplim5} holds because for stage $ 2 $, the \kz{corresponding} empirical distribution converges weakly to the true distribution and the recourse function for stage $ 2 $ is actually a linear expression of $ f_2^r $, \kz{which leads to both the existence of the limit and the consistency},  and \eqref{upplim6} to \eqref{upplim7} holds \kz{\jho{by taking} the limit for stage $ 3,\ldots, T $}. Since inequality \eqref{uppineqn2} holds for any feasible solution $ x $, we conclude that 
	\begin{subeqnarray*} 
		&& \limsup_{N\rightarrow \infty}\hat{z}(d_2(N_2),\ldots,d_T(N_T)) \nonumber \\
		& \leq &  \min c_{1}x_{1} +  \E_{\P_2}\Big[ \min c_2 x_2 + \cdots + \E_{\P_T}[ \min  c_T x_T ]  \Big] \\
		& = & z(0).
	\end{subeqnarray*}
	Thus, inequality \eqref{uppgoal} is proved and $ \lim_{N\rightarrow \infty}\hat{z}(d_2(N_2),\ldots,d_T(N_T)) = z(0) $.
	
	The proof for convergence property of optimal solutions follows similar method to that in Proposition \ref{reflimit} and is \kz{thus} omitted here.
\end{proof}

\section{Solution Approach}\label{sec:sddp}
In this section, we first \jho{utilize} a scenario tree  to model the data process $  \xi_2, \dots,\xi_T $ \jho{based on the historical data}, and then \jho{customize} the SDDP approach to solve our RMSP.

We collect the historical data and then generate a finite scenario tree based on these data with two reference distributions $ \P^\ell$ and $ \P^{u-\ell} $. At each stage $  t=2,\dots, T $,  we denote  $ N_t $ as the total number of different scenarios, and let $ p_{ti}^\ell $ \jho{and} $ p_{ti}^{u-\ell} $ be the corresponding reference probabilities for scenario $ \xi_{ti}, i=1,\dots,N_t $.
Therefore, we can rewrite the dynamic programming equations in \eqref{formulation:t} as follows:
\begin{align}\label{formulation:t-data-driven}
\hat{Q}_{ti}(\bar{x}_{t-1}) = \inf_{x_{t }, u_{t}} \Big\{ & c_{t }(\xi_{ti})x_{t }(\xi_{ti}) + (1-P_{t+1}^\ell )u_{t}(\xi_{ti}) + \hat{\Qe}_{t+1}( x_{t }(\xi_{ti}), u_{t}(\xi_{ti}) ): \notag \\
& A_{t}(\xi_{ti})x_{t }(\xi_{ti}) \geq b_{t }(\xi_{ti}) - B_{t }(\xi_{ti})\bar{x}_{t-1} \Big\}, 
\end{align}
for $ i=1,\dots, N_t $, where \jhw{$ \bar{x}_{t-1} $ is the current solution obtained from stage $ t-1 $, and} the corresponding cost-to-go \jho{function} can be written as follows:
\begin{equation}\label{eqn:cost-to-go-data-driven}
\hat{\Qe}_{t+1}( x_{t }(\xi_{tj }), u_{t}(\xi_{tj }) ) = P_{t+1}^\ell \sum_{i=1}^{N_{t+1}}p_{t+1i}^\ell  [\hat{Q}_{t+1i}(x_{t }(\xi_{tj }))] + (P_{t+1}^u-P_{t+1}^\ell )\sum_{i=1}^{N_{t+1}}p_{t+1i}^{u-\ell}[\hat{Q}_{t+1i}(x_{t }(\xi_{tj })) - u_{t}(\xi_{tj}) ]^+.
\end{equation}

At the first stage, \jho{we have}
\begin{align}\label{formulation:t=1-data-driven}
\min_{x_{1}, u_{1}} \qquad & c_{1}x_{1}  +  (1-P_{2}^\ell )u_{1 } + \hat{\Qe}_{2}( x_{1 }, u_{1} ) \notag \\
s.t. \qquad & A_{1} x_{1}  \geq b_{1}.
\end{align}

\jhw{In the following, we use variables $ x_t, u_t $ instead of $ x_{t }(\xi_{tj }), u_{t}(\xi_{tj }) $ for notation brevity, and we denote $ \bar{x}_{t }, \bar{u}_{t} $ as the current solutions for stage $ t $.}
\jho{As} the cost-to-go functions $ \hat{\Qe}_{t+1}( x_{t }, u_{t} ) $ are convex,   we can use supporting hyperplanes to  make lower approximations of the cost-to-go functions. We define $ \mathfrak{Q}_{t+1}(x_{t }, u_{t}) $ as the current approximation of the cost-to-go function at stage $ t+1$ \jho{for each} $ t= 1,\dots,T-1 $. 

To generate a supporting hyperplane for $ \hat{\Qe}_{t+1}( x_{t }, u_{t} ) $ at $ ( \bar{x}_{t }, \bar{u}_{t} ) $ in \eqref{eqn:cost-to-go-data-driven}, we first consider the subdifferential of function $ [\hat{Q}_{t+1i}(x_{t }) - u_{t}]^+ $ at point $(  \bar{x}_{t }, \bar{u}_{t} ) $, which is 
	$$ \partial [\hat{Q}_{t+1i}(x_{t }) - u_{t}]^+  = \left\{\begin{array}{lr}
	[0,0]  & \text{if } \hat{Q}_{t+1i}(\bar{x}_{t }) < \bar{u}_{t},  \\
	\bigcup\limits_{g\in \partial \hat{Q}_{t+1i}(\bar{x}_{t }) } [g,-1]  & \text{if }   \hat{Q}_{t+1i}(\bar{x}_{t }) > \bar{u}_{t} , \\
	\bigcup\limits_{g\in \partial \hat{Q}_{t+1i}(\bar{x}_{t }) , \ld\in [0,1] }[\ld g, -\ld]  & \text{if } \hat{Q}_{t+1i}(\bar{x}_{t }) = \bar{u}_{t}.
	\end{array}\right.
	$$

Applying the chain rule of subdifferentials, the subgradient of $ \hat{\Qe}_{t+1}( x_{t }, u_{t} ) $ at $ ( \bar{x}_{t }, \bar{u}_{t} ) $ is 
	\begin{equation}\label{eqn:subgradient}
\Bigg[ P_{t+1}^\ell \sum_{i=1}^{N_{t+1}}p_{t+1i}^\ell  \mathbf{g}_{t+1i} + (P_{t+1}^u - P_{t+1}^\ell  )\sum_{i\in \mathcal{J}_{t+1} }p_{t+1i}^{u-\ell}\mathbf{g}_{t+1i}, \; -(P_{t+1}^u-P_{t+1}^\ell )|\mathcal{J}_{t+1}|  \Bigg],
	\end{equation}
	where $ \mathbf{g}_{t+1i} \in \partial \hat{Q}_{t+1i}(\bar{x}_{t }), i=1,\dots,N_{t+1}  $ and $ \mathcal{J}_{t+1}:=\{j:  \hat{Q}_{t+1i}(\bar{x}_{t }) > \bar{u}_{t}\} $.

	The supporting hyperplane of $ \hat{\Qe}_{t+1}( x_{t }, u_{t} ) $ at $ ( \bar{x}_{t }, \bar{u}_{t} ) $ is constructed as follows.
	\begin{align}
\mathit{h}_{t+1}({x}_{t}, {u}_{t}) := &	\hat{\Qe}_{t+1}(\bar{x}_{t}, \bar{u}_{t}) +  [P_{t+1}^\ell \sum_{i=1}^{N_{t+1}}p_{t+1i}^\ell  \mathbf{g}_{t+1i} + (P_{t+1}^u-P_{t+1}^\ell )\sum_{i=1}^{N_{t+1}}p_{t+1i}^{u-\ell}\mathbf{g}_{t+1i}](x_t - \bar{x}_t) \notag \\
& -(P_{t+1}^u-P_{t+1}^\ell )|\mathcal{J}_{t+1}| (u_t -\bar{u}_t). \label{eqn:supporting-plane}
	\end{align}

Therefore, we can update the cost-to-go functions by adding the new supporting hyperplane \eqref{eqn:supporting-plane}, \ie  $ \mathfrak{Q}_{t+1}(x_{t }, u_{t}) := \max\{\mathfrak{Q}_{t+1}(x_{t }, u_{t}), \mathit{h}_{t+1}({x}_{t}, {u}_{t})  \} $.

\subsection{Algorithm}

\begin{algorithm}[!h]                    
	\caption{SDDP Method for  Risk-Averse Multi-Stage Stochastic Program}          
	\label{alg:SDDP}                           
	\begin{algorithmic}[1]                    
		\REQUIRE  Collect historical data and construct data sets $ \Omega_t $ for each stage scenarios $ \xi_t $ and obtain the corresponding reference distributions $ \P_t^\ell , \P_t^{u-\ell}, t=2,\dots,T $. Initialize $ \mathfrak{Q}_{t}( x_{t-1 }(\xi_{t-1 }), u_{t-1}(\xi_{t-1 }) ) = 0$ for $ t=2,\dots,T $. Let $ \bar{z} = \infty $ and $ k=0 $.
		\STATE Solve the first-stage problem \eqref{formulation:t=1-data-driven-apprx}; let $ (\hat{x}_1, \hat{u}_1) $ be the optimal solution.
		\STATE Calculate the lower bound  $ \underline{z} $ as in problem \eqref{formulation:t=1-data-driven-apprx}; if $ \bar{z} + \frac{z_{\alpha/2}}{\sqrt{M-1}}\sum_{i=1}^{M}(z_i -\bar{z})^2 -\underline{z} \leq \varepsilon \underline{z} $ or $ k>K $, stop; otherwise, go to step 3. \\ 
		\vspace{3 mm}
		\textbf{Forward Simulation}
		\FOR{$ t=2,\cdots, T $ }	
		\FOR{$ i=1,\cdots, M_t $}
		\STATE Sample a $ \xi_{ti} $ from the data set $ \Omega_t $.
		\STATE Solve the optimization problem \eqref{formulation:t-data-driven-apprx} for stage $ t $, sample $ i $.
		\STATE Store the optimal solution as $ (\hat{x}_{ti}, \hat{u}_{ti}) $.
		\ENDFOR
		\ENDFOR  
		\vspace{3 mm}
		\STATE Calculate the upper bound $ \bar{z} = c_1\hat{x}_1 + (1-P_{2}^\ell )\hat{u}_{1 } + \hat{\mathfrak{v}}_2 $, where $ \hat{\mathfrak{v}}_t := \sum_{i=1}^{M_t}\hat{\mathfrak{v}}_{ti}/ M_t, \forall  t=2,\dots,T$,  $\hat{\mathfrak{v}}_{ti} = P_{t+1}^\ell {p_{ti}^\ell (c_{ti}\hat{x}_{ti} + \hat{\mathfrak{v}}_{t+1 i} )}  + (1- P_{t+1}^\ell  )p_{ti}^{u-\ell}\hat{u}_{ti} + (P_{t+1}^u- P_{t+1}^\ell  ){p_{ti}^{u-\ell}[c_{ti}\hat{x}_{ti} + \hat{\mathfrak{v}}_{t+1 i} -\hat{u}_{ti}]^+} $, $ \forall  t=2,\dots,T, \forall i =1,\dots,M_t  $. And $ z_i = c_1\hat{x}_1 + (1-P_{2}^\ell )\hat{u}_{1 } + \hat{\mathfrak{v}}_{2i}, \forall i =1,\dots,M_t $
		\\
		\vspace{3 mm}
		\textbf{Backward Recursion}
		\FOR{$ t=T,T-1, \cdots, 2 $ }
		\FOR{each trial decision $ (\hat{x}_{t-1i}, \hat{u}_{t-1i}), i=1,\dots,M_t $}
		\FOR{each scenario $ \xi_{tj} , j=1,\dots,N_t$}
		\STATE Solve the optimization problem~\eqref{formulation:t-data-driven-apprx} with the approximation  $ \mathfrak{Q}_{t+1}(x_{t }, u_{t}) $ for $ t, \hat{x}_{t-1i}, \allowbreak \hat{u}_{t-1i}, \xi_{tj} $
		\STATE Let $ \mathbf{g}_{t-1j}^i $ be the multiplier associated to the constraints of problem~\eqref{formulation:t-data-driven-apprx} at the optimal solution.
		\ENDFOR
		\STATE  Construct one supporting hyperplane~\eqref{eqn:supporting-plane} of the approximate risk-averse expected future cost function for stage $ t-1 $ and add it to $  \mathfrak{Q}_{t-1i}(\hat{x}_{t-1i}) $.
		\ENDFOR
		\ENDFOR
		\STATE $ k \leftarrow k+1$.
		\STATE Go to step 2.
	\end{algorithmic}
\end{algorithm}

In this section, we adopt SDDP to solve our RMSP. SDDP was first proposed in \cite{pereira1991multi} \jhw{to solve MSP} and later studied in \cite{shapiro2011analysis} \jhw{for MSP models where the objective function is a convex combination of expectation and CVaR}. There are two parts in SDDP, a forward simulation that generates a simulated solution, and a backward recursion that improves the approximation of the cost-to-go functions at each stage.
The detailed algorithm is provided in Algorithm~\ref{alg:SDDP}. We describe the application of SDDP to our RMSP in the following section.

In the backward recursion steps, we solve the optimization problem~\eqref{formulation:t-data-driven} with the approximation  $ \mathfrak{Q}_{t+1}(x_{t }, u_{t}) $ instead of the cost-to-go function $ \hat{\Qe}_{t+1}( x_{t }, u_{t} ) $.
\begin{align}\label{formulation:t-data-driven-apprx}
\hat{Q}_{ti}(\bar{x}_{t-1}) = \inf_{x_{t }, u_{t}} \Big\{ & c_{t }(\xi_{ti})x_{t }(\xi_{ti}) + (1-P_{t+1}^\ell )u_{t}(\xi_{ti}) + \mathfrak{Q}_{t+1}( x_{t }(\xi_{ti}), u_{t}(\xi_{ti}) ): \notag \\
& A_{t}(\xi_{ti})x_{t }(\xi_{ti}) \geq b_{t }(\xi_{ti}) - B_{t }(\xi_{ti})\bar{x}_{t-1} \Big\}, 
\end{align}
for $ i=1,\dots, N_t $, where the corresponding approximated expected cost-to-go functions can be updated by adding supporting hyperplanes \eqref{eqn:supporting-plane}.

At the first stage,
\begin{align}\label{formulation:t=1-data-driven-apprx}
\min_{x_{1}, u_{1}} \qquad & c_{1}x_{1}  +  (1-P_{2}^\ell )u_{1 } + \mathfrak{Q}_{2}( x_{1 }, u_{1} ) \notag \\
s.t. \qquad & A_{1} x_{1}  \geq b_{1},
\end{align}
whose optimal \jhw{objective value} is used as the lower bound of our problem.

The forward simulation steps are performed by sampling independent scenarios from the historical data and computing the corresponding optimal value, which will  be used later to calculate the upper bound $ \bar{z} = c_1\hat{x}_1 + (1-P_{2}^\ell )\hat{u}_{1 } + \hat{\mathfrak{v}}_2 $, 
where $$ \hat{\mathfrak{v}}_t =\dfrac{1}{M_t}  \sum_{i=1}^{M_t} \left\lbrace  P_{t+1}^\ell (c_{ti}\hat{x}_{ti} + \hat{\mathfrak{v}}_{t+1 i})  + (1- P_{t+1}^\ell   )\hat{u}_{ti} + (P_{t+1}^u- P_{t+1}^\ell  ){[c_{ti}\hat{x}_{ti} + \hat{\mathfrak{v}}_{t+1 i} -\hat{u}_{ti}]^+} \right\rbrace, t=2,\dots,T.  $$

\subsection{Convergence Property}
In this subsection, we provide the convergence of the proposed algorithm.

\begin{proposition}
	When the basic optimal solutions are employed in the backward steps, the forward step procedure generates the optimal solution for the risk-averse multistage stochastic program w.p.1 after a sufficiently large number of backward and forward steps of the algorithm.
\end{proposition}
\beginproof
Since the total number of scenarios generated from the historical data is finite and the forward steps generate the sample scenarios independently, the conditions in Proposition 3.1 in \cite{shapiro2011analysis} hold. Thus, our proposition holds. 
\endproof

\section{Risk-averse Hydrothermal Scheduling Problem}\label{sec:compute}
In this section, we investigate a risk-averse hydrothermal scheduling problem with distributional \kz{robustness} by applying our proposed reformulation and algorithm in Sections~\ref{sec:reformulation} -- \ref{sec:sddp}.

The risk-averse hydrothermal scheduling problem is aiming to determine an optimal operational schedule for the hydrothermal system \kz{which} minimize\kz{s} the total expected costs including operation cost, fuel cost, and penalty cost for failing to satisfy the electricity load. The hydrothermal system is a combination of hydroelectric generators and other fuel-costing generators like thermal and nuclear generators. 
Hydroelectric generators utilize stored water in the system reservoirs to provide energy, \jho{whose generation is determined} by the water inflow. 
The inflow is assumed to be the only uncertainty of this problem, which follows some unknown distribution and we only have historical data of the inflow volume at each stage.
The risk-averse hydrothermal scheduling problem can be modeled as a stochastic dynamic program if the inflow volume of the current stage is predicted at the beginning of that stage.

\subsection{Problem Formulation}
To describe the problem, we let $ T $ represent the planning horizon, and 
we denote sets of  thermal generators and reservoirs as $  \mathcal{G} $ and $ \mathcal{R}  $, respectively.
For each thermal generator $ g \in \mathcal{G} $, we label the maximum (minimum) generation as $ \overline{G}^g(\underline{G}^g) $, and denote the unit generation cost as $ c_t^g $ at time $ t $.
For each reservoir $ r\in \mathcal{R} $, we denote the maximum (minimum) reservoir storage as $ \overline{R}^r (\underline{R}^r) $, and denote the inflow to reservoir $ r $ at time $ t$ as $ I^r_t $. For convenience, we use energy units, \ie MWh, as the units of $ \overline{R}^r (\underline{R}^r) $ and $ I^r_t $, which can easily be done by multiplying  constant coefficients of the true storage level and inflow amount, respectively.
Furthermore, we use $ D_t $ to represent the total electricity load, and  use $ c^p_t $ to represent the penalty cost for each unit of unsatisfied load at time $ t $.
 
Next, we define five continuous decision variables using energy units MWh as follows. We define 
$ x^g_t $ as the thermal energy generated by generator $ g\in \mathcal{G} $  at time $ t $, and $ x^p_t $ as the energy amount failed to satisfy at time $ t $. For each reservoir $ r\in \mathcal{R} $, we define  
the hydroelectric energy generation at time $ t $ as $ x^r_t $,
the water spillage as $ s_t^r $, and  the storage level as $ v^r_t $ for reservoir $ r\in \mathcal{R} $ at time $ t $, where the initial storage $ v^r_0 $ is given as a parameter. The generated hydroelectricity energy is a linear function of the water outflow, which can be represented by the outflow volume times a constant. 

Based on the above notation, we provide the formulation as follows.
\begin{align}\slabel{hydro-formulation}
\min_{x_{1}} \qquad  & (\sum_{g\in\mathcal{G}}c_{1}^g x_{1}^g +  c_{1}^p x_{1}^p) +  \sup_{\P_2 \in \mathcal{D}^2}\E_{\xi_2\sim \P_2}\Big[ \min_{ x_2} (c_2^g(\xi_2)x_2^g(\xi_2)  +  c_2^p(\xi_2)x_2^p(\xi_2))  \notag \\
&+ \cdots + \sup_{\P_T \in \mathcal{D}^T}\E_{\xi_T\sim \P_T}[ \min_{ x_T} (c_T^g(\xi_T)x_T^g(\xi_T)  +  c_T^p(\xi_T)x_T^p(\xi_T))   ]  \Big]   \label{eqn:hydro-obj} \\
s.t. \qquad & \sum_{r\in\mathcal{R}}x_t^r + \sum_{g\in\mathcal{G}}x_t^g + x_t^p = D_t, \forall t =1, \dots,T, \label{eqn:hydro-load-balance} \\
& v_t^r + x_t^r + s_t^r = v_{t-1}^r + I_t^r, \forall t =1, \dots,T, \forall r \in \mathcal{R}, \label{eqn:hydro-hydraulic} \\
& \underline{G}^g \leq x_t^g \leq \overline{G}^g, \forall t =1, \dots,T, \forall g \in \mathcal{G}, \label{eqn:hydro-thermal-bounds} \\
& \underline{R}^r \leq v_t^r \leq \overline{R}^r, \forall t =1, \dots,T, \forall r \in \mathcal{R}, \label{eqn:hydro-storage-bounds} \\
& x_t^p, x_t^h, s_t \geq 0, \forall t =1, \dots,T, \label{eqn:hydro-nonnegative}
\end{align}
where the objective function~\eqref{eqn:hydro-obj} is to minimize long-term expected thermal generation cost and penalty cost under the worst-case distribution within the ambiguity set. Constraints~\eqref{eqn:hydro-load-balance} ensure the electricity load balance. Constraints~\eqref{eqn:hydro-hydraulic} restrict on the reservoir water balance, where the inflow amount equals to the sum of storage level difference, spillage and water outflow for hydroelectricity generation. Constraints~\eqref{eqn:hydro-thermal-bounds} and \eqref{eqn:hydro-storage-bounds} represent \jho{capacities of} the thermal generation and the reservoir storage, respectively. Constraints~\eqref{eqn:hydro-nonnegative} are the nonnegative constraints of decision variables.

\subsection{Experiment Settings}

In the following, we perform numerical experiments on the risk-averse hydrothermal scheduling problem with distributional \kz{robustness} by applying our proposed reformulation and algorithm in Sections~\ref{sec:reformulation}-\ref{sec:sddp}.  We randomly generate four cases by assuming the inflow model follows four different classes of true probability distributions, \ie lognormal, truncated normal, Weibull, and exponential distributions. For each case at each stage, we first create and collect $ N $ independent random samples from the true distribution to estimate the empirical distribution $ f_0 $, and then construct a scenario tree with $ S $ scenarios at each stage.
Next, the ambiguity set is constructed as $ \mathcal{D} = \{f\geq 0: f_0^r -d \leq f^r \leq f_0^r+d, \forall r = 1, \dots, R, \sum_{r=1}^{R}f^r =1   \} $, where $ d = \max \{\jho{z_{\alpha/2}} \sqrt{f^r_0(1-f_0^r)}/\sqrt{N}, \forall r = 1,\dots,R \} $ \jho{and $ z_{\alpha/2} $ is the $ \alpha $-level z-score (here we select $ \alpha= 5\% $, and $ z_{\alpha/2} = 1.96 $)}.
Finally, we compare our risk-averse optimal solutions with risk-neutral solutions under perfect information. All experiments were coded in C++ and implemented on a computer node with two AMD Opteron 2378 Quad Core Processors at 2.4GHz and 4GB memory. IBM ILOG CPLEX 12.3 is utilized as the linear programming solver. 

For the parameter setting of Algorithm~\ref{alg:SDDP}, we set the optimality gap as $ 5\% $, the iteration limit $ K $ as $ 300 $, the sample number at each stage $ M_t, \forall t =2,\dots,T $ as $ 6 $. 
For uncertainty set parameter settings,  
we create a scenario tree with $ 12 $ different scenarios at each stage, \ie $ S=12 $. For the hydrothermal scheduling problem  parameter settings, we consider a $ 52 $-week planning horizon for a single reservoir and a single thermal generator. We let the reservoir maximum capacity be $ \overline{R} =  10^6 \mbox{MWh} $ and minimum capacity be $ \underline{R} = 10^5 \mbox{MWh} $. The initial storage is set as $ v_0 = 5.5 \times 10^5 \mbox{MWh} $. The thermal unit generation costs at each stage are randomly generated which vary from $\$ 45 /\mbox{MWh}$ to $\$ 85 /\mbox{MWh} $, and the unit penalty cost is set as $\$ 1000 /\mbox{MWh} $.

\subsection{Computational Results}
\begin{figure}[h]
	\begin{subfigure}{.5\linewidth}
		\centering
		\includegraphics[width=8cm]{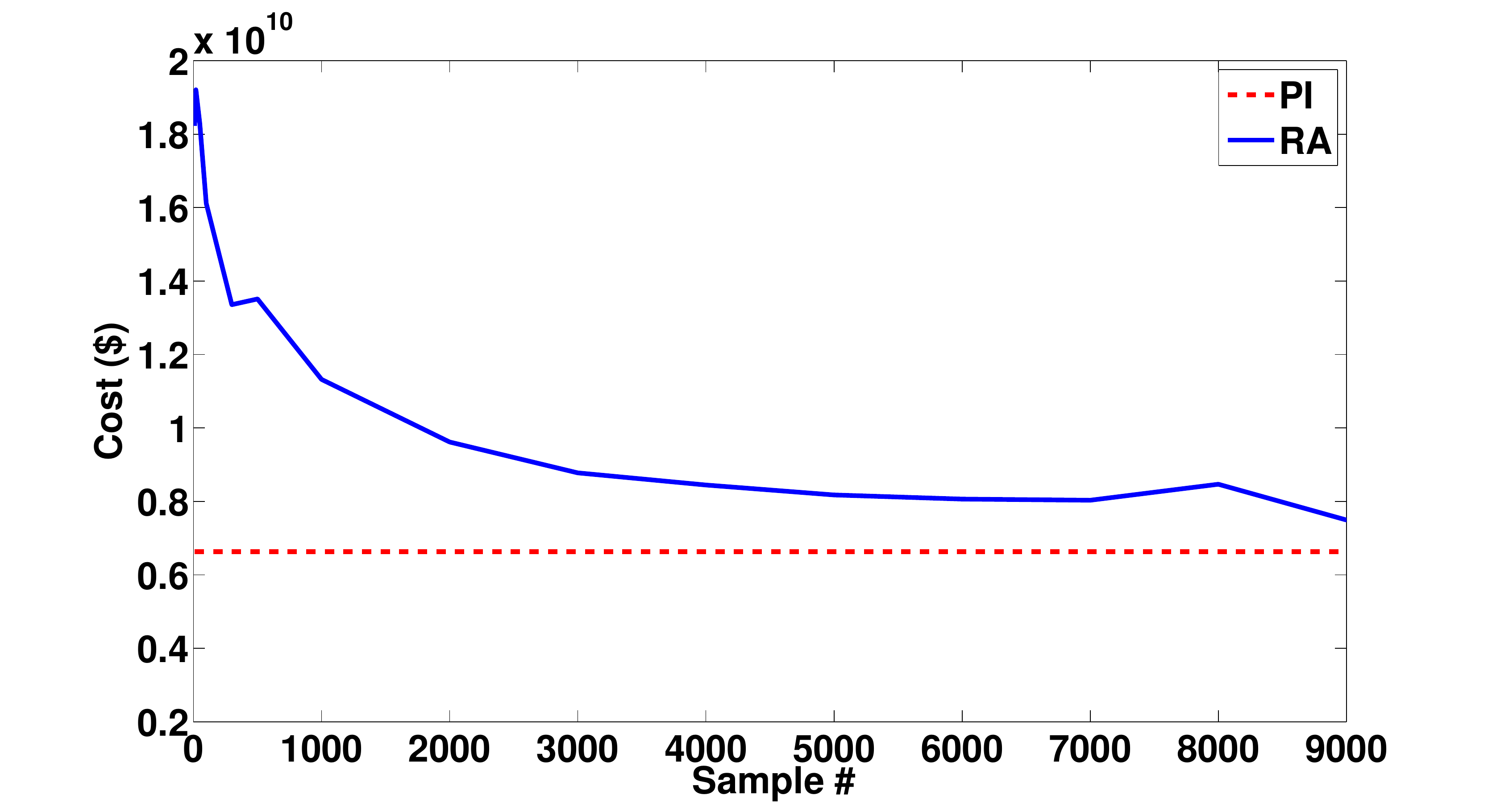}
		\caption{lognormal}
		\label{fig:log-pi-ra}
	\end{subfigure}
	\begin{subfigure}{.5\linewidth}
	\centering
	\includegraphics[width=8cm]{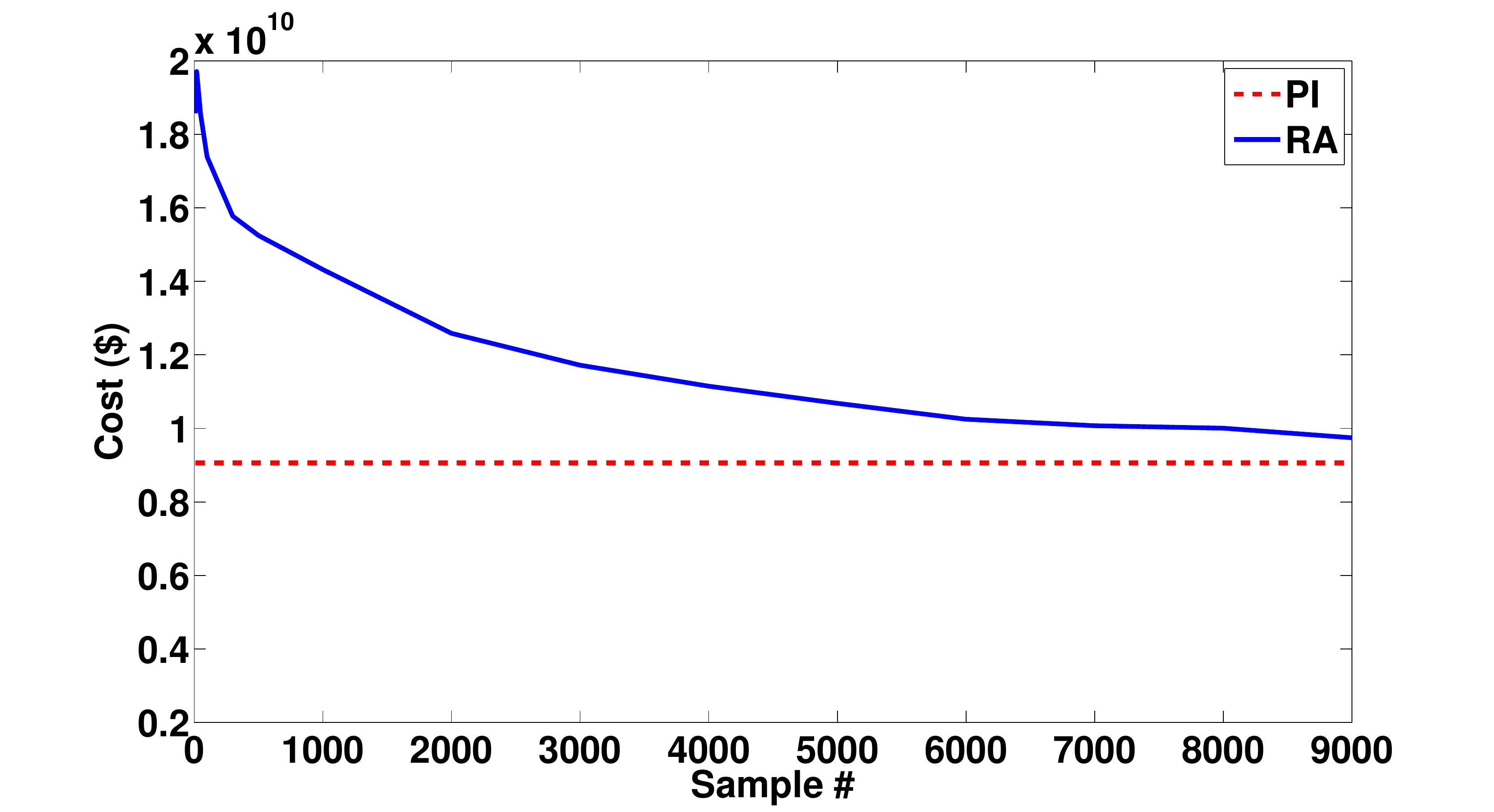}
	\caption{exponential}
	\label{fig:expo-pi-ra}
\end{subfigure}	\\
	\begin{subfigure}{.5\linewidth}
		\centering
		\includegraphics[width=8cm]{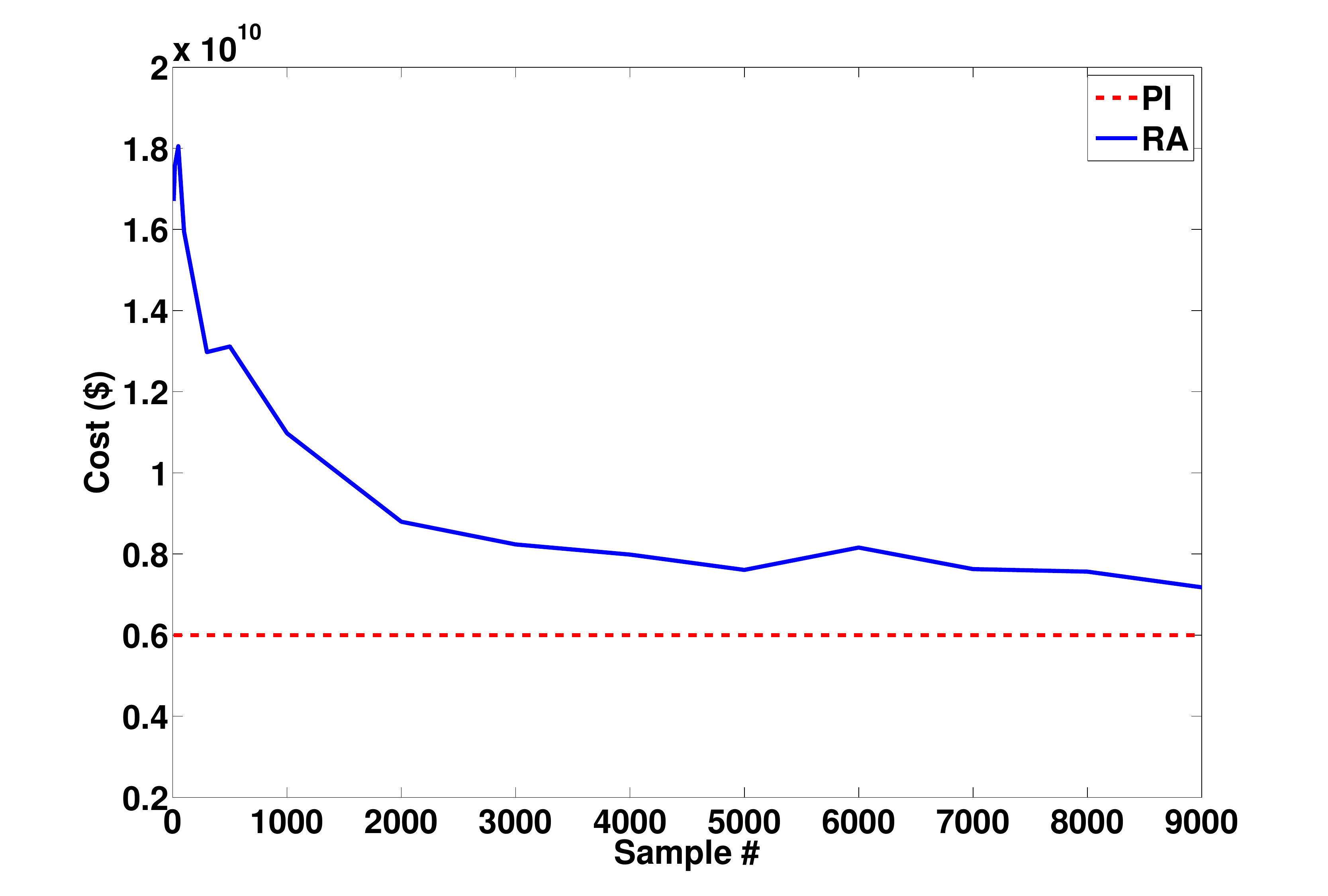}
		\caption{normal}
		\label{fig:normal-pi-ra}
	\end{subfigure}
	\begin{subfigure}{.5\linewidth}
		\centering
		\includegraphics[width=8cm]{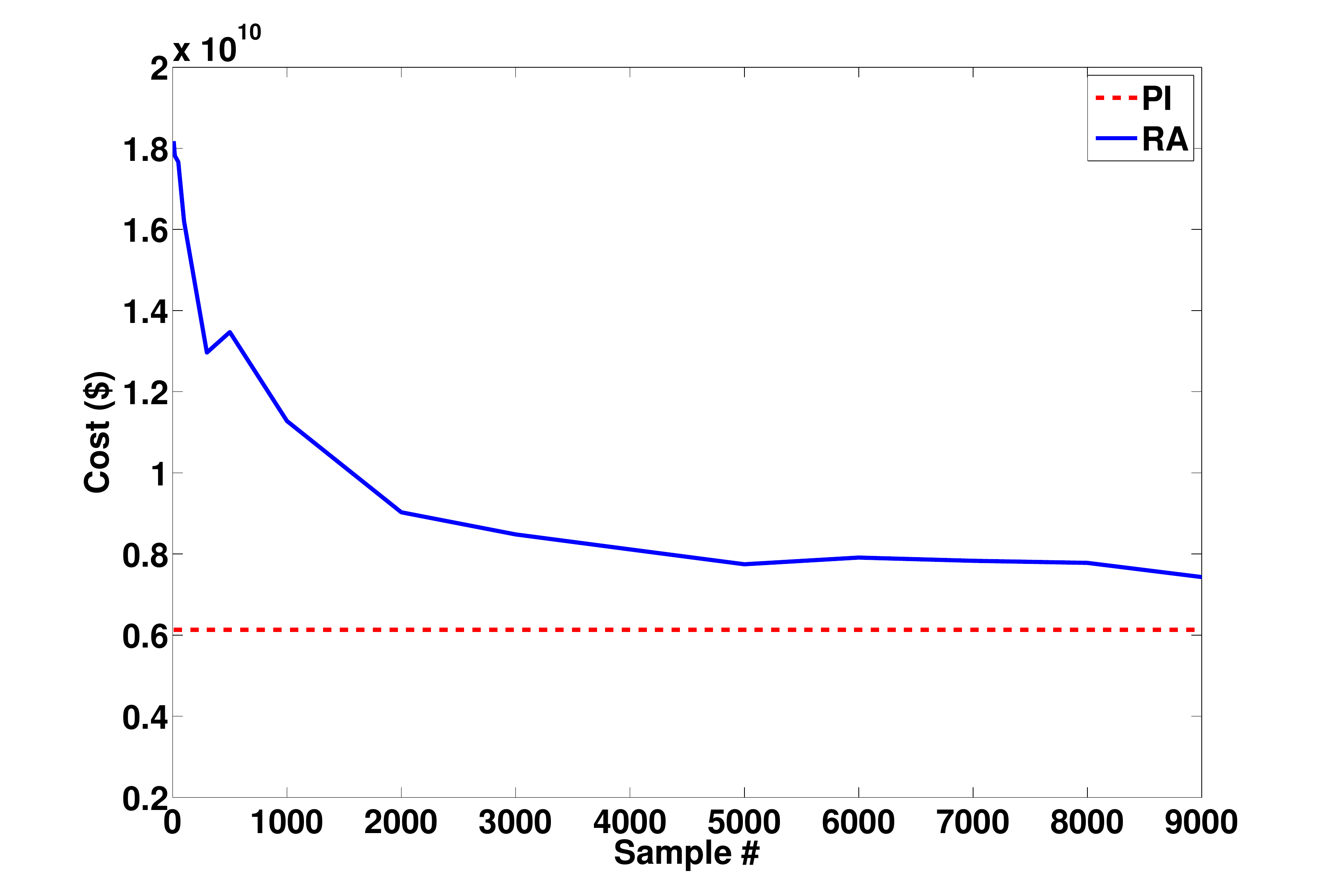}
		\caption{Weibull}
		\label{fig:weibull-pi-ra}
	\end{subfigure}	
	\caption{Comparisons of the risk-averse optimal solution (RA) and the risk-neutral solution with perfect information (PI)}
	\label{fig:sddp-pi-ra}
\end{figure}

\begin{table}[htbp]
	\centering
	\caption{Risk-Averse Solutions Performance}
	\begin{tabular}{c|cc|cc|cc|cc}
		\hline
		Sample & \multicolumn{2}{c|}{normal} & \multicolumn{2}{c|}{exponential} & \multicolumn{2}{c|}{Weibull} & \multicolumn{2}{c}{lognormal} \\
		\hline
		$ N $     & Gap   & Step  & Gap   & Step  & Gap   & Step  & Gap   & Step \\
		\hline
		10    & 1.78  & 300   & 1.05  & 300   & 1.96  & 300   & 1.75  & 300 \\
		20    & 1.93  & 300   & 1.18  & 300   & 1.90  & 300   & 1.90  & 300 \\
		50    & 2.01  & 300   & 1.05  & 300   & 1.88  & 300   & 1.76  & 300 \\
		100   & 1.66  & 300   & 0.92  & 300   & 1.64  & 300   & 1.43  & 300 \\
		300   & 1.16  & 300   & 0.74  & 300   & 1.11  & 300   & 1.01  & 300 \\
		500   & 1.19  & 299   & 0.68  & 300   & 1.20  & 299   & 1.04  & 300 \\
		1000  & 0.83  & 198   & 0.58  & 245   & 0.84  & 198   & 0.71  & 198 \\
		2000  & 0.47  & 120   & 0.39  & 170   & 0.47  & 120   & 0.45  & 130 \\
		3000  & 0.37  & 101   & 0.29  & 141   & 0.38  & 101   & 0.32  & 101 \\
		4000  & 0.33  & 92    & 0.23  & 114   & 0.32  & 92    & 0.27  & 90 \\
		5000  & 0.27  & 85    & 0.18  & 106   & 0.26  & 85    & 0.23  & 85 \\
		6000  & 0.36  & 98    & 0.13  & 84    & 0.29  & 86    & 0.22  & 82 \\
		7000  & 0.27  & 81    & 0.11  & 77    & 0.28  & 86    & 0.21  & 81 \\
		8000  & 0.26  & 84    & 0.10  & 81    & 0.27  & 80    & 0.28  & 91 \\
		9000  & 0.20  & 76    & 0.08  & 74    & 0.21  & 74    & 0.13  & 67 \\
		\hline
	\end{tabular}%
	\label{table:sddp-gap}%
\end{table}%

\begin{figure}[!h]
	\begin{subfigure}{.5\linewidth}
		\centering
		\includegraphics[width=8cm]{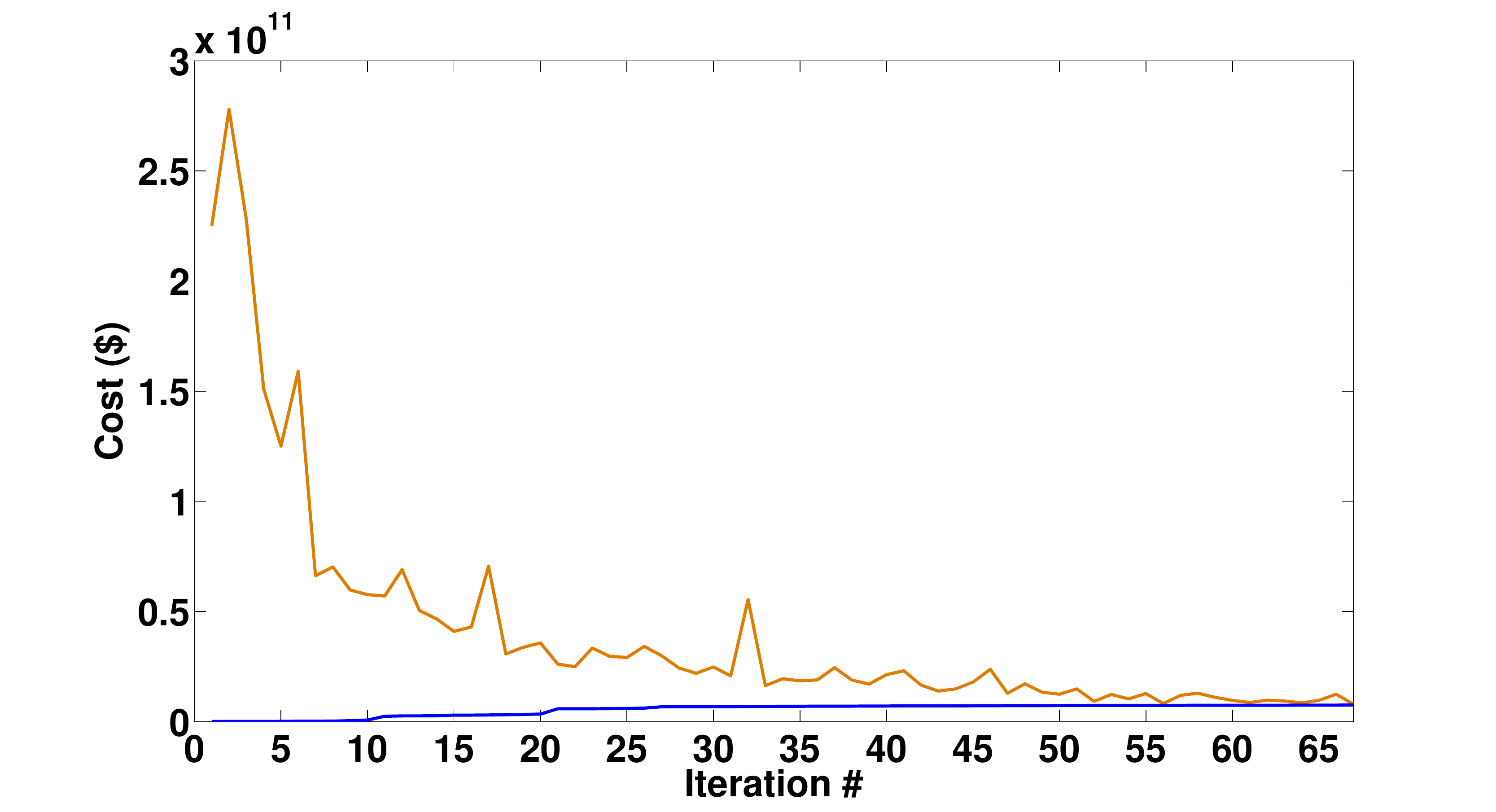}
		\caption{lognormal}
		\label{fig:log-result}
	\end{subfigure}
	\begin{subfigure}{.5\linewidth}
		\centering
		\includegraphics[width=8cm]{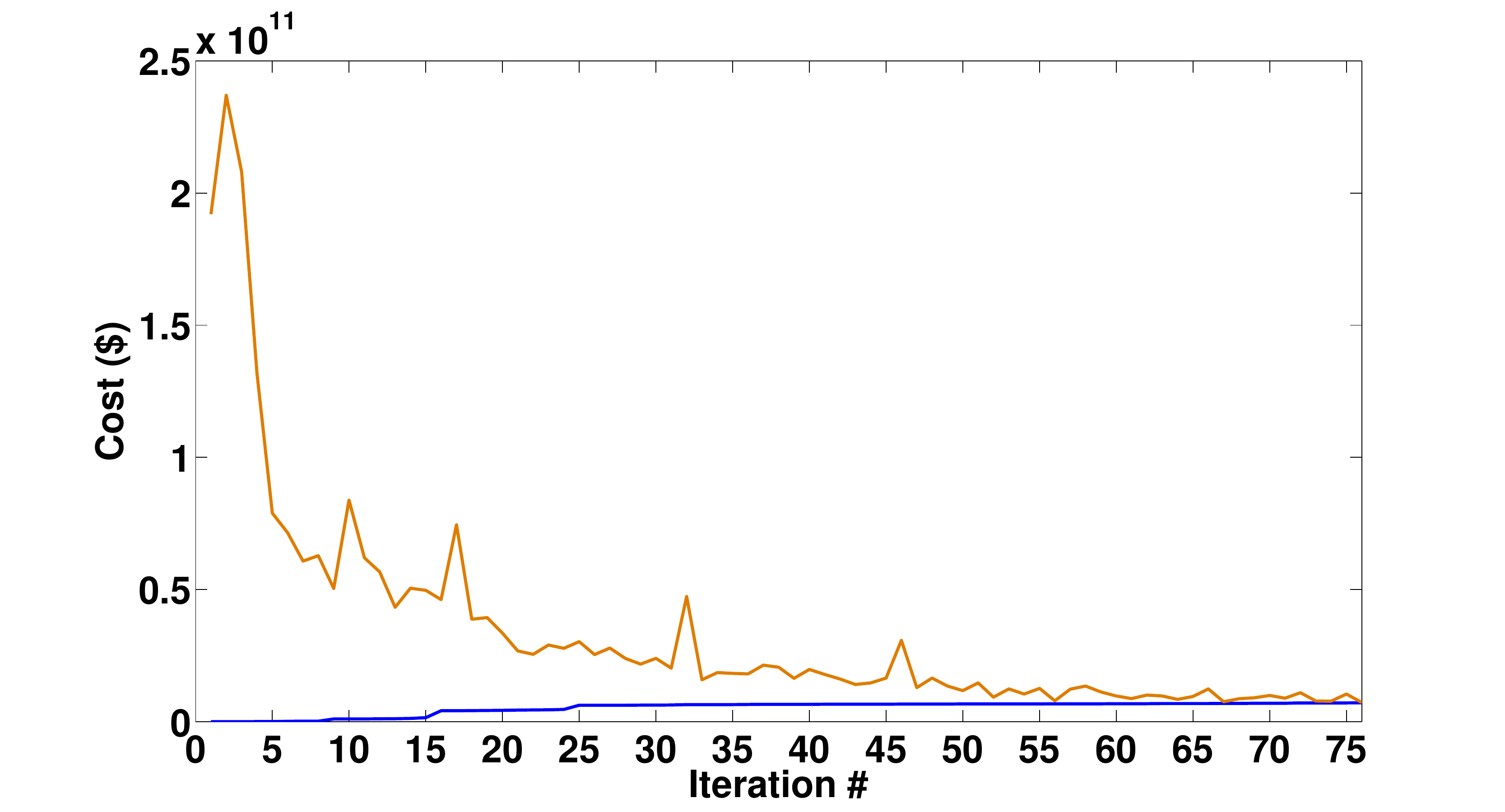}
		\caption{normal}
		\label{fig:normal-result}
	\end{subfigure}\\[1ex]
	\begin{subfigure}{.5\linewidth}
		\centering
		\includegraphics[width=8cm]{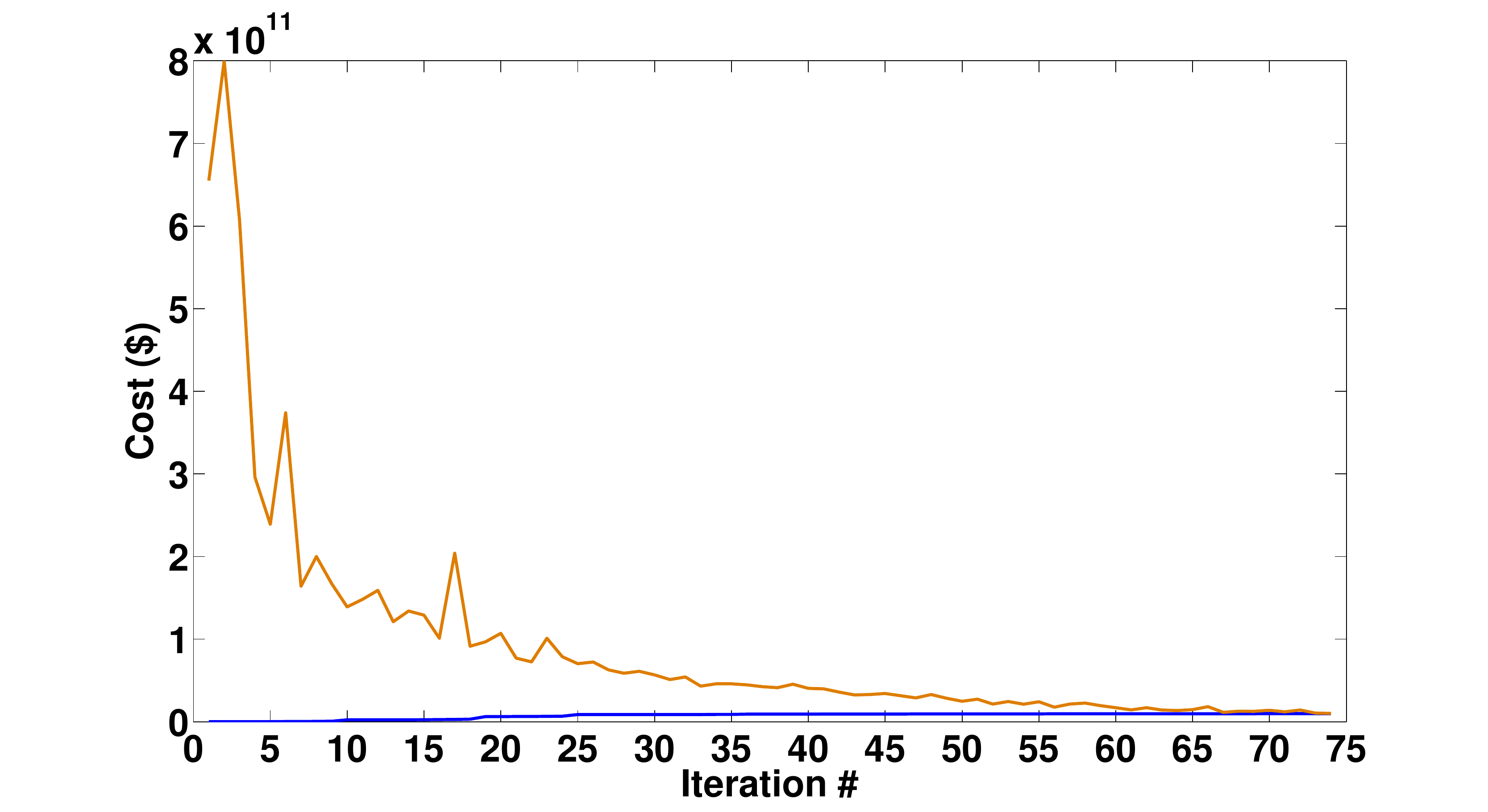}
		\caption{exponential}
		\label{fig:expo-result}
	\end{subfigure}	
	\begin{subfigure}{.5\linewidth}
		\centering
		\includegraphics[width=8cm]{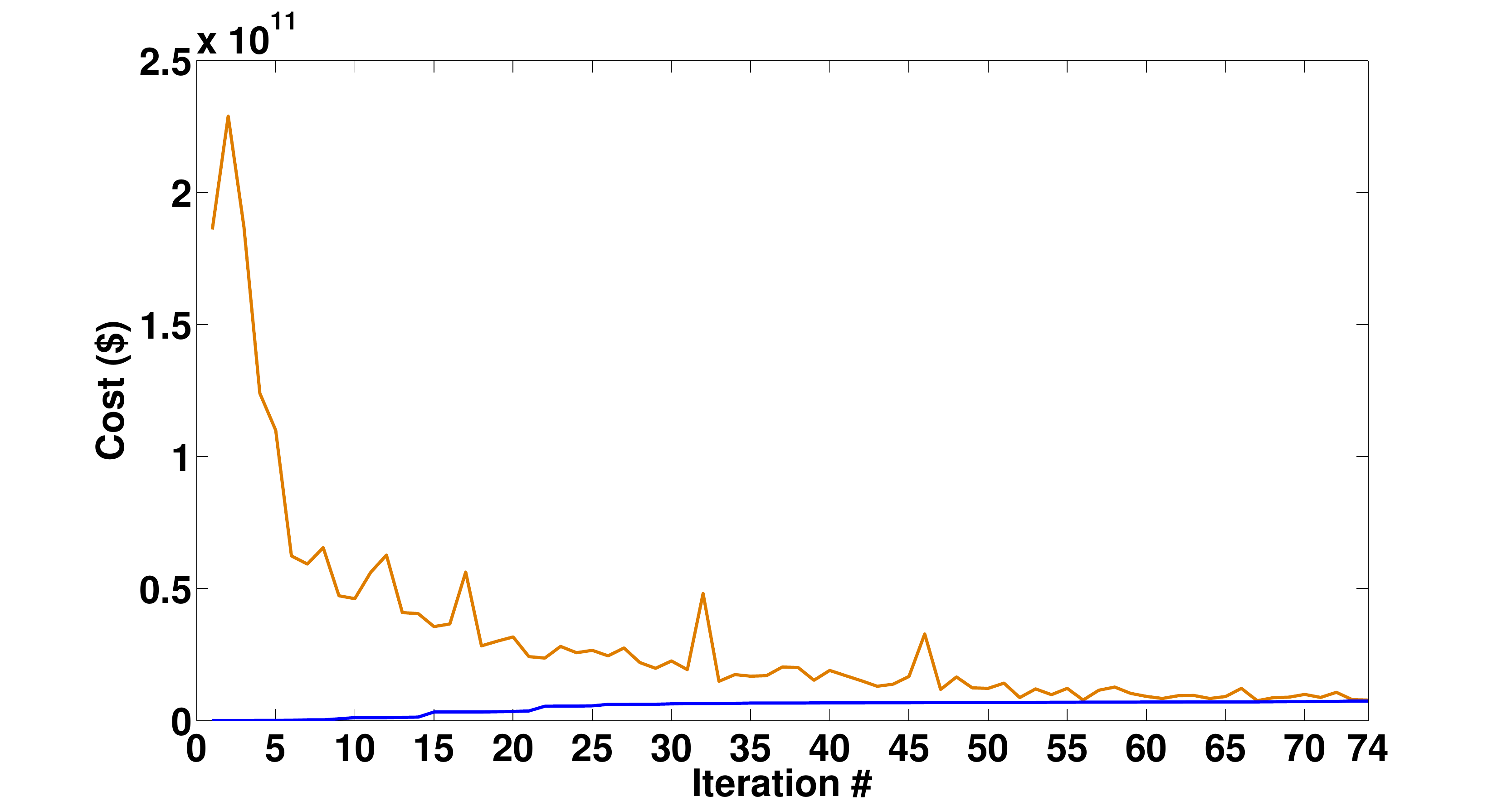}
		\caption{Weibull}
		\label{fig:weibull-result}
	\end{subfigure}	
	\caption{SDDP Convergence Results  with 9000 Data Samples: Iteration - Cost}
	\label{fig:9k-sddp-result}
\end{figure}

\begin{figure}[!h]
	\begin{subfigure}{.5\linewidth}
		\centering
		\includegraphics[width=8cm]{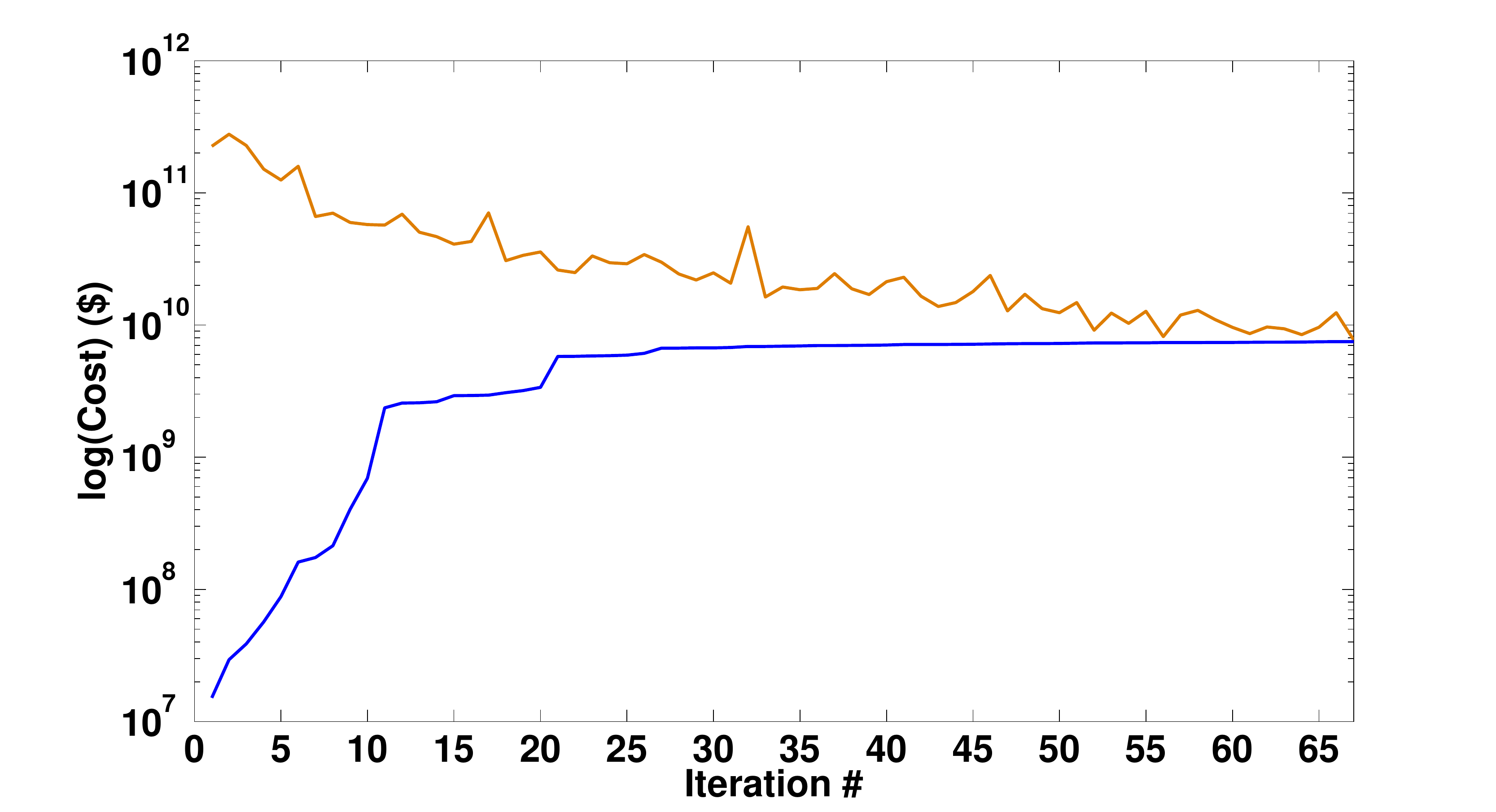}
		\caption{lognormal}
		\label{fig:log-lg-result}
	\end{subfigure}
	\begin{subfigure}{.5\linewidth}
		\centering
		\includegraphics[width=8cm]{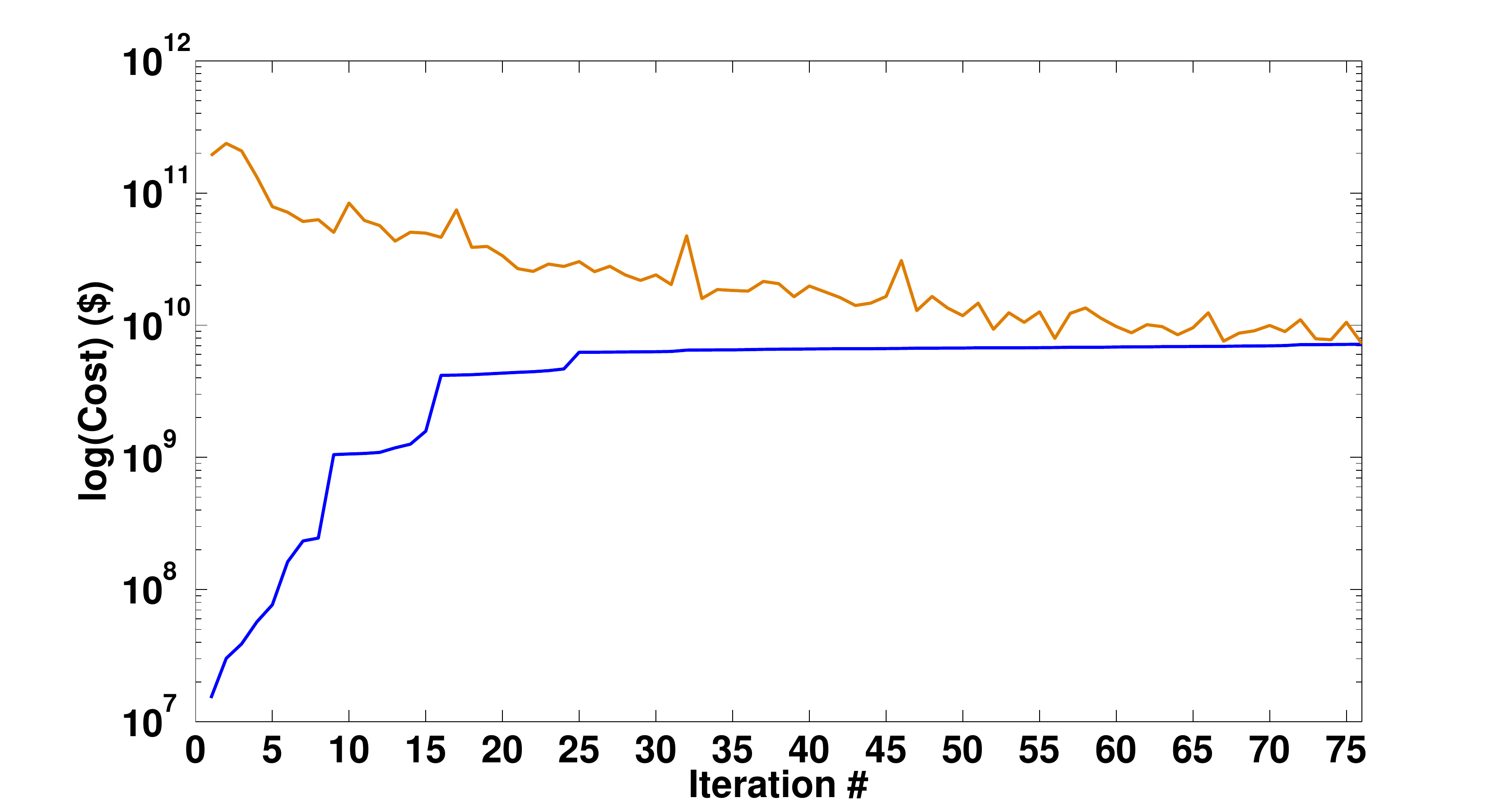}
		\caption{normal}
		\label{fig:normal-lg-result}
	\end{subfigure}\\[1ex]
	\begin{subfigure}{.5\linewidth}
		\centering
		\includegraphics[width=8cm]{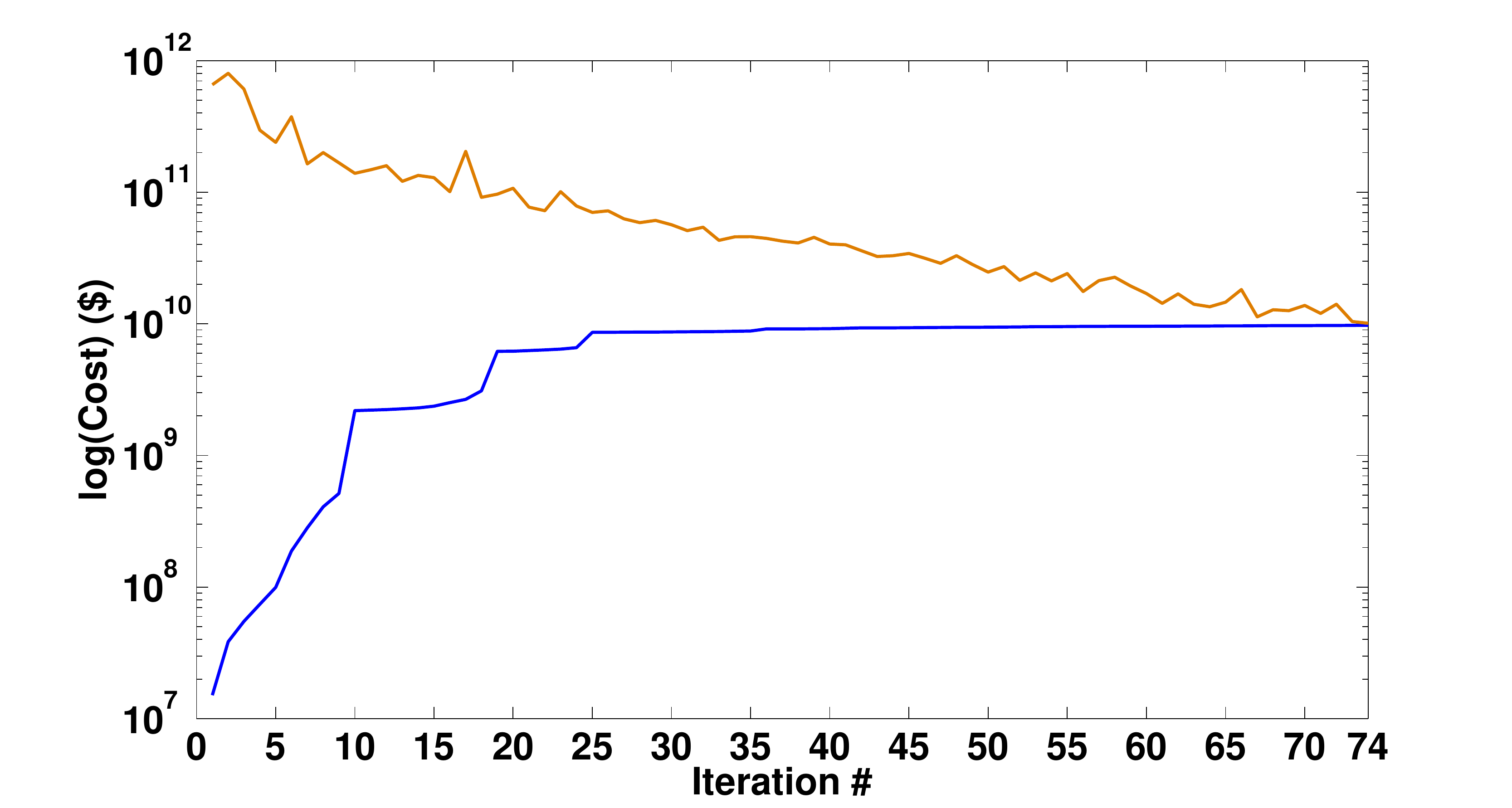}
		\caption{exponential}
		\label{fig:expo-lg-result}
	\end{subfigure}	
	\begin{subfigure}{.5\linewidth}
		\centering
		\includegraphics[width=8cm]{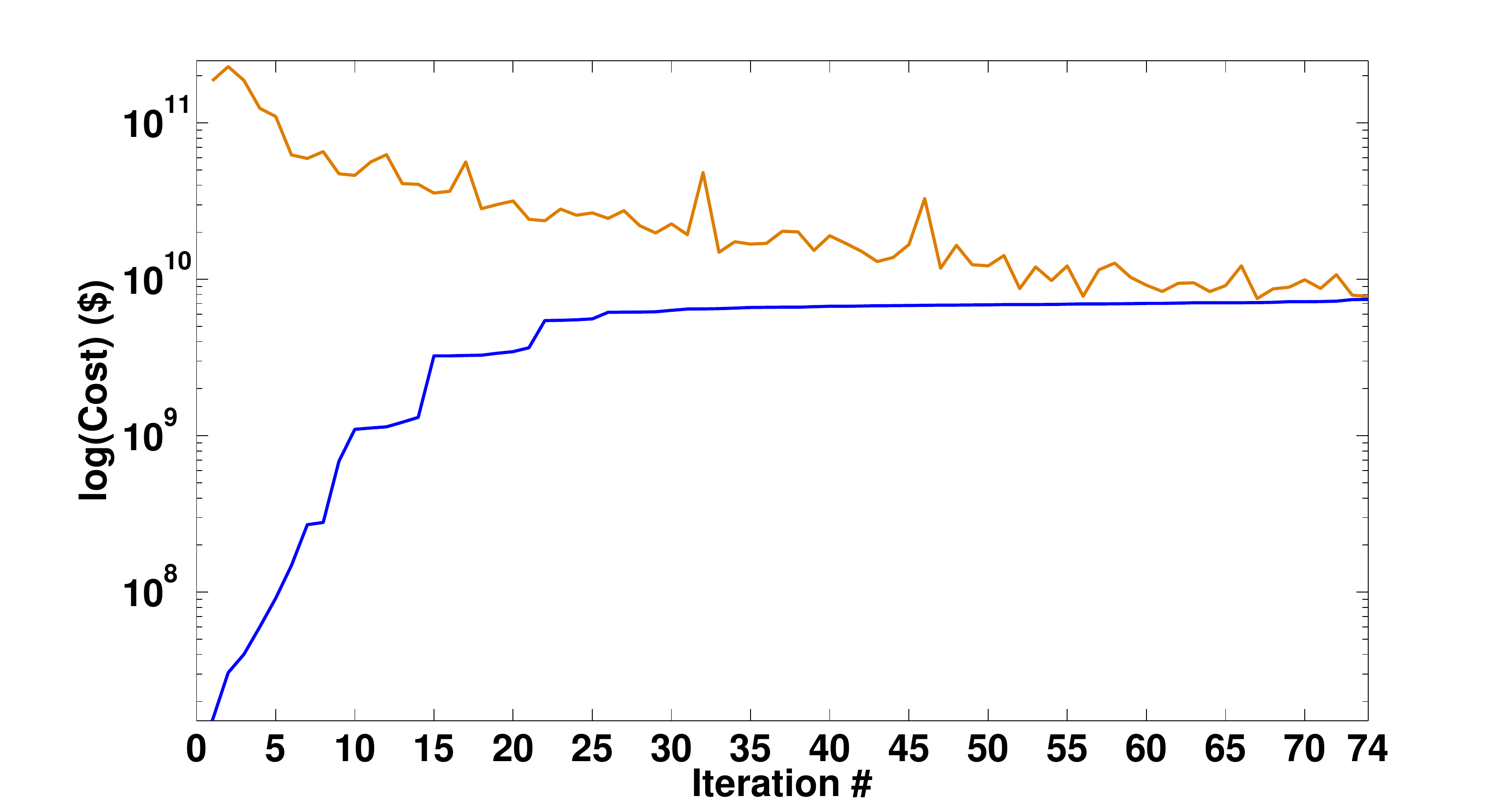}
		\caption{Weibull}
		\label{fig:weibull-lg-result}
	\end{subfigure}	
	\caption{SDDP Convergence Results with 9000 Data Samples: Iteration - log(Cost)}
	\label{fig:9k-sddp-lg-result}
\end{figure}
In the following, we compare our risk-averse solutions and their corresponding risk-neutral solutions with perfect information under various data sample sizes and true distribution settings in Figure~\ref{fig:sddp-pi-ra}. 
As indicated from Figure~\ref{fig:sddp-pi-ra}, our risk-averse solutions converges to their corresponding risk-neutral solution as the number collected data samples increases, and the convergence process evolves moderately quickly after collecting $ 2000 $ data samples, which numerically proves the convergence of our risk-averse multistage stochastic program with distributional ambiguity to the risk-neutral multistage stochastic program.

We present the solution gaps between the risk-averse (RA) objective value and the risk-neutral (PI) objective value, and the algorithm iteration steps under various data sample size and distribution settings  in Table~\ref{table:sddp-gap}. In this table, the column ``Gap" represents the gap between RA and PI, \ie $ \mbox{Gap} = (Z_{a} - Z_{n})/Z_{n} $, where $ Z_{a} $ is the risk-averse solution obtained from Algorithm~\ref{alg:SDDP}, and $ Z_{n} $ is the risk-neutral objective value under the corresponding true distribution. The column ``Step" represents the iteration steps when our algorithm stops, where $ 300 $ means that the corresponding case ceases due to the predefined iteration step limit. 
As indicated from Table~\ref{table:sddp-gap}, the solution gap decreases and the number of required iteration step becomes smaller as the number collected data samples increases, which is coincident with the convergence of the ambiguity set. That is, with more data samples collected, our ambiguity set size shrinks and thus it is faster to solve the risk-averse stochastic program over this ambiguity set.

We provide the algorithm performance under various distribution settings with 9000 data samples in Figure~\ref{fig:9k-sddp-result}. Figures \eqref{fig:log-result} - \eqref{fig:weibull-result} represent the upper bound and lower bound evolving processes under different true distributions lognormal, normal, exponential and Weibull, respectively. The horizontal axis is the number of performed iteration steps and the vertical axis is the total cost. Since the gap between upper bound and lower bound is significant in the first 30 iteration steps in the Figure~\ref{fig:9k-sddp-result}, we provide a closer look at the convergent behavior of Algorithm~\ref{alg:SDDP} in Figure~\ref{fig:9k-sddp-lg-result}, by using the log value of total cost in the vertical axis. The algorithm converges quickly in $ 35 $ steps and terminates with about $ 75 $ steps for each distribution setting.

\section{Conclusion}\label{sec:conclusion}
In this paper, we present an equivalent reformulation of RMSP, where we use a convex combination of expectation and CVaR to replace the worst-case expectation. The reformulation prevents repeating min-max patterns in the multistage program. As the size of collected data samples goes to infinity, we show that RMSP converges to the risk-neutral MSP, where the optimal objective value and the set of optimal solutions of RMSP converge to those of risk-neutral MSP. We adopt the SDDP algorithm to solve the reformulated RMSP and provide the convergence property for the algorithm. To test the RMSP computation performance, we implement numerical experiments for the risk-averse hydrothermal scheduling problem under different true distributions, which demonstrate the convergence of our RMSP to risk-neutral MSP as the collected data increase to infinity.

\baselineskip=12pt
\bibliographystyle{plain}
\bibliography{RAMSP-arxiv1}

\end{document}